\documentclass[12pt]{article}
\usepackage{xcolor}
\usepackage[a4paper,scale=.7,centering]{geometry}
\usepackage{amssymb,amsmath,amsfonts,amsthm,commath}
\usepackage{mathrsfs}
\usepackage{bbm}
\usepackage{float}

\usepackage[comma,authoryear]{natbib}

\newif\ifShowMarginPar
\ShowMarginPartrue

\newcommand{\N}{\mathbb N}
\newcommand{\R}{\mathbb R}
\newcommand{\Zz}{{\mathbb Z}^2}

\renewcommand{\P}{\mathbb P}

\newcommand{\E}{\mathbb E}

\newcommand{\Ai}{{\cal A}}
\newcommand{\Bi}{{\cal B}}
\newcommand{\Ci}{{\cal C}} 
\newcommand{\Di}{{\cal D}}
\newcommand{\Ei}{{\cal E}}

\newcommand{\Gi}{{\cal G}}
\newcommand{\Hi}{{\cal H}}

\newcommand{\Pc}{{\cal P}}

\newcommand{\Ri}{{\cal R}}
\newcommand{\Si}{{\cal S}}
\newcommand{\Ti}{{\cal T}}

\newcommand{\Vi}{{\cal V}}

\newtheorem{theorem}{Theorem}[section]
\newtheorem{lemma}{Lemma}[section]

\newtheorem{prop}{Proposition}[section]
\newtheorem{defn}{Definition}[section]
\newtheorem{assumption}{Assumption}[section]
\newtheorem{cor}{Corollary}[section]

\newtheoremstyle{remark}{3pt}{3pt}{}{}{\bfseries}{:}{.5em}{}
\theoremstyle{remark}
\newtheorem{remark}{Remark}[section]

\title{Criticality and Covered Area Fraction in Confetti and Voronoi Percolation}
\author{ {Partha Pratim Ghosh and  Rahul Roy}
	\footnote{E-Mail: {\tt p.pratim.10.93@gmail.com, rahul@isid.ac.in}}\\
	{\it Indian Statistical Institute, New Delhi}}

\begin{document}
	\maketitle
	
	\begin{abstract}
		Using the randomized algorithm method developed by  Duminil-Copin, Raoufi and Tassion (2019b), we exhibit sharp phase transition for the confetti percolation model. This provides an alternate proof, than that of Ahlberg, Tassion and Texeira (2018), for the critical parameter for percolation in this model to be $1/2$ when the radius of the underlying shapes for the distinct colours arise from the same distribution. In addition, we study the covered area fraction for this model, which is akin to the covered volume fraction in continuum percolation. Modulo a certain `transitivity condition', this study allows us to calculate exact critical parameter for percolation when the underlying shapes for different colours may be of different sizes. Similar results are also obtained for the Poisson Voronoi percolation model when different coloured points have different growth speeds.
	\end{abstract}
	
	\vspace{0.1in}
	
	\vspace{0.1in}
	
	\section{Introduction}
	\label{Sec:intro}
	
	In this paper we study two models of percolation viz., confetti and Voronoi percolation. The confetti percolation model was introduced by Benjamini and Schramm (1998), while the Voronoi percolation model was introduced by Vahidi-Asl and Wierman (1990). The interest in these models arise from the fact that, in a sense, they are self-dual models, akin to bond percolation of the square lattice. Historically, these models have been studied by stochastic geometers for a long time.  Matheron (1968)  introduced the confetti model as a stochastic model of  sequential superposition of random sets to study natural images and their occlusion. M{\o}ller (1994) studied the random Voronoi tessellations obtained when the seeds of the tessellations arise as a Poisson point process in space.

	In this section we describe the confetti percolation model, the Voronoi percolation is introduced in Section~\ref{Voronoi}.
	For $\lambda\in [0,1]$, let 
	\begin{itemize}
		\item[(i)] $\Pc:= (X_1, X_2, \ldots) $  be a Poisson process of intensity $1$ on $ \R^2 \times (0,\infty)$ defined on a probability space $(\Xi_1, \Hi_1, P_1)$,
		\item[(ii)] $\boldsymbol {\rho} := (\rho_1, \rho_2, \ldots)$ be  a collection of i.i.d. positive, bounded random variables defined on a probability space $(\Xi_2, \Hi_2, P_2)$,
		\item[(iii)] $\boldsymbol {\beta} := (\beta_1, \beta_2, \ldots)$ be a collection of i.i.d. positive, bounded  random variables defined on a probability space $(\Xi_3, \Hi_3, P_3)$,
		\item[(iv)] $\boldsymbol {\varepsilon} := (\varepsilon_1, \varepsilon_2, \ldots )$  be a collection of i.i.d. Bernoulli random variables taking values $1$ and $0$ with probabilities $\lambda$ and $1- \lambda$ respectively defined on a probability space $(\Xi_4, \Hi_4, P_4)$.
	\end{itemize}
	We assume the four processes above are independent of each other and defined on the product space
	$(\Xi, \Hi, \P_\lambda)$ where
	 $\Xi:= \Xi_1\times \Xi_2\times \Xi_3\times \Xi_4$, 
	$ \Hi := \Hi_1 \otimes \Hi_2 \otimes \Hi_3\otimes \Hi_4$ and 	$ \P_\lambda := P_1 \otimes P_2 \otimes P_3\otimes P_4$.
	
	Let 
	\begin{align}
		\label{r:defRB}
		&{\bf R}_i := X_i + ([-\rho_i/2, \rho_i/2]^2 \times \{0\}), \quad {\bf B}_i := X_i + ([-\beta_i/2, \beta_i/2]^2 \times \{0\} ),\nonumber \\
		&\Ci:= \left(\cup_{\{i: \varepsilon_i = 1\}} {\bf R}_i\right) \cup  \left(\cup_{\{i: \varepsilon_i = 0\}} {\bf B}_i \right).
	\end{align}
	Thus the space $ \R^2 \times (0,\infty)$ consists of `floating squares', with the square centred at $X_i$ being either a `red' square with sides of length $\rho_i$ or a `blue' square  with sides of length $\beta_i$ according as $\varepsilon_i $ equals $1$ or $0$ respectively.
	We denote this model by $(\Pc, \lambda, \rho, \beta)$.
	
	A point  ${\bf x} = (x_1, x_2) \in \R^2$ is coloured red if a ray sent vertically up from $(x_1, x_2, 0) $ hits a red `floating' square first, and it is coloured blue if it hits a blue `floating' square  first. Note that with probability $1$, the ray must hit a `floating' square. Formally, for a point  ${\bf x} = (x_1, x_2) \in \R^2$, let
	\begin{align}
		h({\bf x})&:= \inf\{s: (x_1, x_2, s)\in \Ci\} \nonumber\\
		i({\bf x})&:= j \text{ where } h({\bf x}) = X_j(3)
	\end{align}
	where $ X_j(3)$ denotes the third coordinate of $X_j$. The point ${\bf x} = (x_1, x_2) \in \R^2$ is coloured  red if $\varepsilon_{i({\bf x})} =1$, and it is coloured blue if $\varepsilon_{i({\bf x})} =0$. Note that by the properties of Poisson processes, $i({\bf x})$ is almost surely unique and finite, and thus every ${\bf x} \in \R^2$ is coloured red or blue almost surely.

	Confetti percolation studies questions of red/blue percolation on the $x-y$ plane $\R^2$, i.e., for what values of the parameters 
	$\lambda$, $\beta$, and  $\rho$ do we have an unbounded connected red component or an unbounded connected blue component (or possibly both/neither).
	
	Let $C^{\text{red}}$ and $C^{\text{blue}}$ be the (random) red and blue regions respectively of $\R^2$ in the model $(\Pc, \lambda, \rho, \beta)$. Also, let $C^{\text{red}}({\bf 0})$ and $C^{\text{blue}}({\bf 0})$ be, respectively,  the maximal connected components of $C^{\text{red}}$ and $C^{\text{blue}}$ containing the origin ${\bf 0}$. Let
	$$
	\lambda_c (\rho,\beta) := \inf\{\lambda : \P_\lambda(\text{diam} (C^{\text{red}}({\bf 0})) = \infty) > 0\},$$
	here, for a set $S \subset \R^2$, $\text{diam}(S) := \sup\{|x-y| : x,y\in S\}$, where $|\cdot |$ denotes the standard Euclidean norm on $\R^2$. We note here that $\lambda_c (\rho,\beta)$ and $\lambda_c (\beta,\rho)$ are different, with the latter relating to the critical parameter obtained when the red (respectively blue) Poisson points have squares of sides $\beta$ (respectively $\rho$) attached to them. 
	From ergodicity, we immediately have 
	$$
	\lambda_c (\rho,\beta) = \inf\{\lambda : \P_\lambda(C^{\text{red}} \text { has an unbounded connected component}) =1\}.
	$$
	\begin{remark}
	\label{nonremark}
	Corollary~\ref{Nontrivial} shows that for all $\rho$ and $\beta$ positive bounded random variables  we have 
\begin{equation}
\label{nonT}
0 < \lambda_c(\rho,\beta) < 1.
\end{equation}
\end{remark}
	
	Benjamini and Schramm (1998) conjectured that when $\beta = \rho = \text{constant}$
	$$
	\lambda_c (\rho,\beta) =  1/2.
	$$
	Using methods developed by Bollob\'{a}s and Riordan (2006a) for Voronoi percolation,  Hirsch (2015) proved the conjecture when the underlying shape is a square.
	Later M\"{u}ller (2017) extended the proof when the underlying shape is a disc. Finally Ahlberg, Tassion and Texeira (2018) showed that the conjecture is true when $\beta$ and $\rho$ are random variables with the same distribution which need not be bounded, but light tailed.

	Bordenave, Gousseau and Roueff (2006), Galerne and Gousseau (2012)  studied geometric properties of this model where they used the term
	`dead leaves model' for this confetti percolation model. 	\begin{assumption}
		\label{r:ass}
		Throughout this paper, we assume that $\rho$ and $\beta$  are random variables such that $0 < \rho, \beta \leq R$ for some $R > 0$.
	\end{assumption}
	
	Our first result, whose proof is based on methods developed by Duminil-Copin, Raoufi and Tassion (2019a),  is regarding the critical parameter $\lambda_c (\rho,\beta) $ when both the red and blue squares have random sizes. This provides an alternate proof to Theorem~8.3 of Ahlberg, Tassion and Texeira (2018) albeit in the case when the random shapes are bounded.

	\begin{theorem}
		\label{th:dual}
		For $\rho$ and $\beta$ bounded random variables, we have
		$$
		\lambda_c(\rho,\beta) + \lambda_c(\beta, \rho) = 1.
		$$
		In particular, when $\rho \stackrel{d}{=} \beta$, due to symmetry we have
		$$
		\lambda_c(\rho,\beta) =1/2.
		$$	    
	\end{theorem} 
	
	To prove the above theorem, we obtain a sharp phase transition of this model. Ahlberg, Tassion and Teixeira (2018) obtains a similar sharp threshold result for both the confetti model and Voronoi model (discussed in the Section~\ref{Voronoi}) on $\R^2$ via Russo-Seymour-Welsh method, which, in addition to the results in Theorem \ref{th:dual}, allows them to show that neither red nor blue percolates at  criticality. The randomized algorithm method developed by Duminil-Copin, Raoufi and Tassion (2019b),  which we apply, simplifies the proof of Ahlberg, Tassion and Teixeira (2018) and also is valid for any dimension, although we state it only for $\R^2$. Also our results go through {\it mutatis mutandis} for discs, regular polygons and any measure preserving affine transformations of these shapes.
	
	Let $S_r:=\left\{z\in \mathbb{R}^2\,:\, \|z\|_{\infty}=r \right\}$, $B(y,r):=\left\{z\in \mathbb{R}^2 \,:\, \|z-y\|_{\infty}\leq \frac{r}{2} \right\}$,  
	$$
	\theta_n(\lambda):=\mathbb{P}_\lambda(C^{\text{red}}({\bf 0}) \cap S_n \neq \emptyset)\text{ and }
	\theta(\lambda):= \P_\lambda(\text{diam} (C^{\text{red}}({\bf 0})) = \infty).
	$$
	Theorem~\ref{th:dual} will follow from the following proposition, which is proved via the randomized algorithm method developed of Duminil-Copin, Raoufi and Tassion (2019b).  
	\begin{prop}
		\label{theo1}
		For any $\lambda<\lambda_c(\rho,\beta) $, there exists $c_\lambda>0$ such that for any $n\geq 1$,
		\[ \theta_n(\lambda)\leq \exp\left(-c_\lambda n\right).   \] 
		Furthermore, there exists $c>0$ such that  for any $\lambda_c(\rho,\beta) < \lambda $,
		\[\theta(\lambda)\geq c(\lambda -\lambda_c(\rho,\beta) ).\]      
	\end{prop}
We remark here that recently Last, Peccati and Yogeshwaran (2021) have a more direct method of proving the above proposition than the discretisation method presented here.

Covered volume fraction, introduced by physicists, was believed to be the parameter governing percolation  for the continuum Boolean percolation model (see Meester and Roy (1996)). We study the analogous covered area fraction $\text{CAF}_{\text{red}}(\lambda, \rho,\beta)$ for the confetti model where
$$
\text{CAF}_{\text{red}}(\lambda, \rho,\beta) := \E_\lambda\left[\ell(\text{red region in } [0,1]^2)\right],
$$
here $\ell$ denotes the Lebesgue measure on $\R^2$.
The critical covered area fraction for the confetti model is defined as 
$$
\text{cCAF}_{\text{red}}( \rho,\beta) := \E_{\lambda_c(\rho,\beta)}\left[\ell(\text{red region in } [0,1]^2)\right].
$$

Whenever the sizes of the underlying cubes/spheres  of a continuum Boolean percolation model are fixed (not random), a simple scaling argument shows that the critical covered volume fraction is a constant irrespective of the sizes of the cubes/spheres, however if the sizes of the  cubes/spheres are  random, then the critical covered volume fraction may differ (see Meester, Roy and Sarkar (1994) and Gou\'{e}r\'{e} and Marchand (2016)).

For the confetti model, when $\rho \overset{d}{=} \beta$, from Theorem~\ref{th:dual}  we have  $\text{cCAF}_{\text{red}}( \rho,\beta) = 1/2$. Theorem \ref{th:cCAFrandom} below shows that when $\rho$ and $\beta$ are random and with different distributions, $\text{cCAF}_{\text{red}}( \rho,\beta)$ may be different from $1/2$.

Also an analogous result (to that of the universality of the critical covered volume fraction for the continuum Boolean percolation model when the underlying shapes are fixed)  for the confetti model 
would be that $\text{cCAF}_{\text{red}}( \rho,\beta) =  1/2$ when $\rho$ and $\beta$ are fixed, not necessarily equal. 
 Unfortunately we have not been able to prove this directly, and we need an intuitively appealing condition to show that it holds. 
 In case this condition holds, Theorem \ref{th:cCAF1/2} below provides us  exact values of the critical parameter $\lambda_c(\rho,\beta)$ for red and blue squares of fixed, but not necessarily same sizes. To state the condition, it is convenient to use an equivalent construction of the confetti model. 
Let $\Ri$ and $\Bi$ be two independent Poisson point processes on $\R^2 \times (0,\infty)$ of intensities $\lambda_r$ and $\lambda_b$ respectively. At points of $\Ri$ we place red squares of side length $\rho$ and at points of $\Bi$ we place blue squares of side length $\beta$ as in~\eqref{r:defRB}. We denote this model by $(\Ri, \lambda_r, \rho; \Bi, \lambda_b, \beta)$ and we note that, on scaling this model is equivalent to the model $(\Pc, \frac{\lambda_r}{\lambda_r + \lambda_b}, \rho,\beta)$.

\vspace{.5em}\noindent {\bf Transitivity condition:} 
\textit{Let $\rho$, $\beta$  and $\gamma$ be three positive fixed constants (not necessarily
	equal), and let  $\lambda_r$, $\lambda_b$ and $\lambda_g$ be the intensities of the three independent Poisson point processes red, blue and green, labeled $\Ri$, $\Bi$ and $\Gi$, respectively. Suppose red is  supercritical in the red/blue confetti model  $(\Ri, \lambda_r, \rho; \Bi, \lambda_b, \beta)$ and blue is supercritical in the blue/green confetti model $(\Bi, \lambda_b, \beta; \Gi, \lambda_g,  \gamma )$, then red is supercritical in the red/green confetti model $(\Ri, \lambda_r, \rho; \Gi, \lambda_g, \gamma )$.}\vspace{.5em} 

The transitivity condition is equivalent to
$$
\frac{\lambda_r}{\lambda_r+\lambda_b}> \lambda_c(\rho,\beta) \text{ and }
\frac{\lambda_b}{\lambda_b+\lambda_g}> \lambda_c(\beta,\gamma)
\Rightarrow\frac{\lambda_r}{\lambda_r+\lambda_g}> \lambda_c(\rho,\gamma) ,
$$
or equivalently,
\begin{align}
	\label{r-TC}
	\left(\frac{1}{\lambda_c(\rho,\beta)}-1\right)\left(\frac{1}{\lambda_c(\beta,\gamma)}-1\right)\leq \frac{1}{\lambda_c(\rho,\gamma)}-1,
\end{align}
for any $\rho$, $\beta$ and $\gamma$ constants, not necessarily equal.

	We believe that the transitivity condition holds and simulations suggest the veracity of our belief. More details, the source code etc. of the simulation are available at  \texttt{https://www.isid.ac.in/\char`\~rahul/index.php/source-code/}. Also, at the end of Section~\ref{Sec:caf}, we present a sufficient condition (Proposition \ref{r:suff}) to check the validity of the transitivity condition, along with a discussion on how this sufficient condition arose.

\begin{theorem}
	\label{th:cCAF1/2}
	The following are equivalent:
	\begin{itemize}
		\item[(i)] The transitivity condition~\eqref{r-TC} holds. 
		\item[(ii)] For any $\rho$, $\beta$ positive constants the critical covered area fraction for $(\Pc, \lambda, \rho,  \beta)$ equals $1/2$.
		\item[(iii)] For any $\rho$, $\beta$ positive constants in the confetti model $(\Pc, \lambda, \rho,  \beta)$  red percolates if  $\lambda > \frac{\beta^2}{\rho^2 + \beta^2}$ and blue percolates if  $\lambda< \frac{\beta^2}{\rho^2 + \beta^2}$, i.e., 
		$\lambda_c(\rho,\beta) = \frac{\beta^2}{\rho^2 + \beta^2}$.
	\end{itemize}
\end{theorem}  

Also, as in the case of continuum percolation, for $\rho$ and $\beta$ not fixed constants, we have that the covered area fraction for red could be large without red percolating.

\begin{theorem}
	\label{th:cCAFrandom}
	For any $t \in (0,1)$, there exists a confetti model $(\Pc, \lambda, \rho, \beta)$, with $\rho$ and $\beta$ random,   for which $CAF_{\text{red}}(\lambda,\rho,\beta) > t$ but red does not percolate.
\end{theorem}

In the next four sections, we obtain the results for the confetti model. In Section~\ref{Voronoi}, we introduce the Voronoi model and provide a sketch of the modifications required for similar results of the Voronoi model.

	\section{Auxiliary results}
	\label{Sec:exp}
	
	Let $R_n$ be the rectangle $ [0,n]\times[0,3n] $ and $H_n$ be the event that there exists a horizontal red crossing of $R_n$, i.e., $R_n\cap C^{\text{red}}$ contains a connected component which intersects both the left edge $\{0\} \times[0,3n] $ and the right edge $\{n\} \times[0,3n] $ of the rectangle $R_n$. 
	
	We first have an analogue of Lemma 3.3 of Meester and Roy (1996). The proof being quite similar, we relegate it to the appendix. For $R$ as in Assumption~\ref{r:ass},
	\begin{prop}
		\label{th:exp}
		Consider the confetti  model $(\Pc, \lambda, \rho, \beta)$. There exists a constant $ \kappa_0\in(0,1)$ such that whenever  $\P_\lambda(H_N)<\kappa_0$  for some $N\geq R$,  we have 
		\[
		\P_\lambda(\text{diam} (C^{\text{red}}({\bf 0}))\geq a)\leq c_1\exp(-c_2a)
		\]
		for all $a > 0$ and for some positive constants $c_1$ and $c_2$ depending only on $\kappa_0$. 
	\end{prop}

	Let $\boldsymbol \omega \in \Xi_1\times \Xi_2 \times  \Xi_3$ and (with a slight abuse of notation) let $\boldsymbol\varepsilon \in  \Xi_4$. Thus
	$(\boldsymbol \omega, \boldsymbol\varepsilon)$ is a sample point from the space $\Xi$. 
	
	We take $\boldsymbol \omega$ to stand for three sequences 
	\begin{equation}
		\label{omega}
		\text{\parbox{.885\textwidth}{(i) $(x_1, x_2, \ldots )$, $x_i \in \R^2 \times (0, \infty)$, a realisation of the Poisson point process, \\
				(ii) $(r_1, r_2, \ldots )$ with $r_i > 0$, a realisation of the process $\boldsymbol \rho$, \\
				(iii) $ (b_1, b_2, \ldots )$ with $b_i > 0$, a realisation of the process $\boldsymbol \beta$.
		}}
	\end{equation}

	\begin{defn} 
    	An event $A \in \Hi$ is called \textbf{increasing} if $(\boldsymbol \omega , \boldsymbol \varepsilon)\in A$  implies $(\boldsymbol \omega, \boldsymbol{\varepsilon'})\in A$ for $\boldsymbol {\varepsilon'}\geq \boldsymbol{\varepsilon}$ (i.e.,  $\varepsilon'_j\geq \varepsilon_j$ for all $j\in\N$). $A$ is \textbf{decreasing} if $A^{\complement}$ is increasing.
	\end{defn}
	\noindent For increasing events, standard methods (see Meester and Roy (1996),  Bollob\'{a}s and Riordan (2006b)) yield

	\vspace{.5em}\noindent {\bf FKG inequality:} 
	\textit{For $ A $ and $ B $, both increasing or both decreasing events, we have $ \P_\lambda(A \cap B) \geq
		\P_\lambda(A)\P_\lambda(B) $. }\vspace{.5em} 
		 
	\noindent The following corollary affirms our claim in Remark \ref{nonremark}:	
	
\begin{cor}
\label{Nontrivial} For all $\rho$ and $ \beta$ positive random variables bounded above by $R$,     we have 
$$0<\lambda_c(\rho,\beta)<1.
$$
\end{cor}
\begin{proof}
		We fix $N\geq R$ and choose $\lambda_1\in(0,1)$ small such that $\P_{\lambda_1}(H_N)<\kappa_0$. Here, $\kappa_0$ is as in Proposition~\ref{th:exp}. Therefore by Proposition~\ref{th:exp},
	\[
	\P_{\lambda_1}(\text{diam} (C^{\text{red}}({\bf 0}))\geq a)\leq c_1\exp(-c_2a)
	\]
	for all $a > 0$ and for some positive constants $c_1$ and $c_2$. This immediately implies $\theta(\lambda_1)=0$ which yields $\lambda_c(\rho,\beta)\geq \lambda_1>0$.
	
	Now choose $\lambda_2\in(0,1)$ large enough such that 
	\begin{equation*}
		\P_{\lambda_2} \left(\text{there exists a  horizontal blue crossing of } R_N \right)<\kappa_0.
	\end{equation*}
	Because of the symmetry of the model, Proposition~\ref{th:exp} implies
	\[
	\P_{\lambda_2}(\text{diam} (C^{\text{blue}}({\bf 0}))\geq a)\leq c_3\exp(-c_4a)
	\]
	for all $a > 0$ and for some positive constants $c_3$ and $c_4$. The FKG inequality yields
	\begin{equation*}
		\P_{\lambda_2}(\text{diam} (C^{\text{blue}}([0,1]^2))\geq a)\leq c_5\exp(-c_6a)
	\end{equation*}
	for all $a > 0$ and for some positive constants $c_5$ and $c_6$. Here, for a set $S\subset \R$, $C^{\text{blue}}(S)$ represents the union of all the  connecting components of $C^{\text{blue}}$ that intersect $S$. Thus, for any $n\in\N$,
	\begin{align}
		&\P_{\lambda_2} \left(\text{there exists a  horizontal blue crossing of } R_n \right)\nonumber\\
		\leq\,\,& \sum_{i=1}^{3n} 	\P_{\lambda_2}(\text{diam} (C^{\text{blue}}([0,1]\times[i-1,i]))\geq n) 
		\leq 3nc_5\exp(-c_6n).     \label{Equ:horizontal_blue_crossing}
	\end{align}
	Now, consider the sequence of rectangles $\{Q_i\}_{i\geq 1}$ defined as
	\begin{equation*}
		Q_i=\begin{cases}
			[0,3^i]\times [0,3^{i+1}] & \text{ if $i$ is odd;}\\
			[0,3^{i+1}]\times [0,3^i] & \text{ if $i$ is even.}
		\end{cases}
	\end{equation*}
	By~\eqref{Equ:horizontal_blue_crossing}, we have
	\begin{align*}
		&\sum_{i=1}^{\infty}\P_{\lambda_2} \left(\text{there exists a blue path in $Q_i$ connecting its longest sides} \right)\\
		\leq\,\,& \sum_{i=1}^{\infty} 3^{i+1}c_5\exp(-c_63^i)<\infty.
	\end{align*}
	Therefore, by Borel-Cantelli lemma, there is almost surely a finite number of $i$ for which there is a blue path in $Q_i$ connecting its longest sides, which then implies that with probability 1, for all $i$ except a finitely many, there is a red path in $Q_i$ connecting its shortest sides. By construction of $\{Q_i\}_{i\geq 1}$, this immediately implies that there is an infinite red path in the confetti model $(\Pc,\lambda_2,\rho,\beta)$ almost surely, and hence $\lambda_c(\rho,\beta)\leq \lambda_2<1$.
\end{proof}

	For a fixed $\boldsymbol\omega$ and $a > 0$, let 
	\begin{align*}
		D_{ \boldsymbol\omega }(a) : = &\{i: \text{ there exists  }\mathbf{y}=(y_1,y_2,y_3) \in \left(\mathbf{R}_i\cup\mathbf{B}_i\right) \cap\left([-a,a]^2\times\R  \right)  \\
		& \text{ such that } (y_1, y_2,s) \not\in \mathbf{R}_j\cap\mathbf{B}_j\text{ for any }j 
		\text{ and  all } 0 < s < y_3 \}, 
	\end{align*} 
	where $\mathbf{R}_i$ and $\mathbf{B}_i$ are as in~\eqref{r:defRB}. 
		In other words, $D_{ \boldsymbol\omega }(a)$ is precisely the set of indices of all Poisson points under $\boldsymbol\omega$ whose colour may influence the colouring of $[-a,a]^2$.

	\begin{theorem}
		\label{dchi}
		For any $a>0$, we have $\E\left[ |D_{ \boldsymbol\omega }(a)|  \right]<\infty$.
	\end{theorem}

		\begin{proof}
			Since $\rho, \beta>0$, we can get $\delta > 0$ such that $\alpha=\P(\min\{\rho, \beta\}>\delta)>0$. We Take $Q= [-\delta/4, \delta/4]^2$ and $H$  the height of the lowest point $X_j\in Q\times (0,\infty)$ for which $\min\{\rho_j,\beta_j\}>\delta$. Since such vertices $X_j$ form a Poisson point process of intensity $\frac{\alpha\delta^2}{4}$, we have
			$\E[H] = \frac{4}{\alpha\delta^2}$. 
			
			Now,
			\begin{align*}
				\Big|D_{ \boldsymbol\omega }\Big(\frac{\delta}{4}\Big)\Big|&\leq 1+\left|\left\{ X_i\in Q\times (0,H) : \min\{\rho_i,\beta_i\}\leq\delta\right\} \right|\\
				&\qquad\qquad +\left|\left\{ X_i\in \left([-R-1,R+1]^2\setminus Q\right)\times (0,H) \right\} \right|,
			\end{align*}
			where $R$ is as in Assumption~\ref{r:ass}. Hence
			$$
			\E\Big[ \Big|D_{ \boldsymbol\omega }\Big(\frac{\delta}{4}\Big)\Big| \Big]
			\leq 1+ (1-\alpha)\cdot\frac{\delta^2}{4}\cdot\E[H]+ 4(R+1)^2\cdot\E[H] <\infty.
			$$
			Now, by choosing $\delta$ reciprocal of an integer and then by tiling the square $[-\lceil a\rceil,\lceil a \rceil]^2$ with squares of the form  $[0,\delta/2]^2$, we obtain the theorem.
	\end{proof}

	\begin{defn}	
			An event $A$ is called \textbf{local} if there exist $n\in\N$ such that  $1_A$  is determined by the colouring of $[-n,n]^2$, i.e., $(\boldsymbol{\omega},\boldsymbol{\varepsilon}) \in A$ implies $(\boldsymbol{\omega},\boldsymbol{\varepsilon'}) \in A$ for all $\boldsymbol \varepsilon'$ such that $\varepsilon'_i = \varepsilon_i$ for  $i \in D_{ \boldsymbol\omega }(n)$.
	\end{defn}
	\begin{defn} 
			For $\boldsymbol{\omega}$ as in~\eqref{omega},
			$x_i$ is called \textbf{pivotal} for $(\boldsymbol{\omega},\boldsymbol{\varepsilon},A)$ if $1_A(\boldsymbol{\omega},\boldsymbol{\varepsilon})\neq 1_A(\boldsymbol \omega,\boldsymbol{\varepsilon^{(i)}})$, where $\boldsymbol{\varepsilon^{(i)}}$ is such that $\varepsilon^{(i)}_i=1-\varepsilon_i$ and $\varepsilon^{(i)}_j=\varepsilon_j$ for all $j\neq i$. 
	\end{defn}
	We shall denote the set of all pivital points for an event $A$ by $\text{Piv}_A$. We use the following theorem (see Exercise~19.3 of Last  and Penrose (2018)):
	\begin{theorem}[Russo's formula]
		\label{russo}
		Let $A$ be a local increasing event. Then $\lambda\mapsto\mathbb{P}_\lambda(A)$ is differentiable and
		\[
		\frac{d}{d\lambda}\mathbb{P}_\lambda(A)=\mathbb{E}_\lambda\left[|\text{Piv}_A| \right],
		\]
		where $|\text{Piv}_A|$ stands for the number of pivotal points for the event $A$.	
	\end{theorem}

	We will also need the following two results from Duminil-Copin, Raoufi and Tassion (2019a).
	\begin{theorem}[OSSS inequality]
		\label{osss}
		Let $\left( \Omega^{\infty},\bigotimes_{i\in \mathbb{N}}P \right)$ be a product probability space, and $f: \Omega^{\infty}\rightarrow\{0,1\}$. An algorithm $T$ determining $f$ takes a configuration $\omega=\left(\omega_i\right)_{i\in \mathbb{N}}\in \Omega^{\infty}$ as an input, and it reveals the values of $\omega_i$'s one by one. At each step, which co-ordinate of $\omega$ will be revealed next  depends on the values of the co-ordinates revealed so far. The algorithm stops as soon as $f$ is determined. Assuming that the algorithm $T$ determines $f$ in finite steps almost surely, we have
		\[ \mbox{Var}(f)\leq \sum_{i=1}^{\infty}\delta_i(T)\mbox{Inf}_i(f),  \] 
		where $\delta_i(T)$ and $\mbox{Inf}_i(f)$ are respectively the revealment and the influence of the $i$th co-ordinate defined by
		\begin{align*}
			\delta_i(T)&= \bigotimes_{j\in \mathbb{N}}P \left(T \mbox{ reveals the value of } \omega_i \right)\\
			\mbox{Inf}_i(f)&=\bigotimes_{j\in \mathbb{N}}P \left(f(\omega)\neq f(\omega^{(i)}) \right).		
		\end{align*}
		
		The above $\omega^{(i)}$ represents the random element in $\Omega^{\infty}$  which is the same as $\omega$ in every co-ordinate except the $i$th co-ordinate which is resampled independently. 
	\end{theorem}
	\begin{lemma}
		\label{copin}
			Consider a converging sequence of increasing differentiable functions $f_n:\left(a,b\right)\rightarrow \left(0,M\right)$ satisfying 
			$$
			f_n'\geq \frac{n}{\sum_{k=0}^{n-1}f_k} f_n
			$$
			for all $n\geq1$. Then, there exists $x_0\in \left[a,b\right]$ such that the following holds:
			\begin{itemize}
				\item[(i)] For all $x \in \left(a,x_0\right)$,  there exists $c_x>0$ such that for any $n$ large enough,  $ f_n(x)\leq M\exp(-c_xn)$.
				\item[(ii)] For all $x \in \left(x_0,b\right)$, $f=\lim\limits_{n\rightarrow\infty}f_n$ satisfies $f(x)\geq x-x_0$.
			\end{itemize}
	\end{lemma}

	\section{Proof of Theorem~\ref{th:dual}}
	\label{r:equality}
	
		We first state a result needed for the proof of Proposition~\ref{theo1}. Recall
		$\theta_n(\lambda)=\mathbb{P}_\lambda(C^{\text{red}}({\bf 0}) \cap \delta ([-n,n]^2) \neq \emptyset)$. 
	\begin{lemma}
		\label{lem2}
		For any $\zeta\in (0,1)$, there exists a constant $u_{\zeta}$ such that for all $\lambda\in (0,\zeta)$ and all $n \geq 1$,
		\[
		\frac{d}{d\lambda}\theta_n(\lambda)\geq \frac{n}{\sum_{k=0}^{n-1}\theta_k(\lambda)}\,u_{\zeta}\,\theta_n(\lambda).
		\]
	\end{lemma}

	Now, we prove Proposition~\ref{theo1} using this result.
\begin{proof}[Proof of Proposition~\ref{theo1} using Lemma~\ref{lem2}]
	First note that   $\theta_n(\lambda)$ is increasing in $\lambda$ and converges  pointwise to $\theta(\lambda)$ as $n \to \infty$.
	We now take $\zeta:=(\lambda_c(\rho,\beta)+1)/2$. From Corollary \ref{Nontrivial}, we know that $0 <\zeta < 1$.
	Using Lemma~\ref{lem2}, and taking $ f_k=\theta_k/u_{\zeta} $ and $M=1/u_{\zeta}$, we have from Lemma~\ref{copin},  there exists $q_{\zeta}\in[0,\zeta]$ such that and for all $n$ large enough, we have 
\begin{align}
& \theta_n(\lambda)\leq \exp\left(-c_{\lambda}n\right) , \text{ for all }\lambda\in(0,q_{\zeta}) \text{ and  some } c_{\lambda}\label{prooftheo1:eq1}\\
&\theta(\lambda)\geq {u_{\zeta}}(\lambda-q_{\zeta})\text{ for all }\lambda\in(q_{\zeta}, \zeta).\label{prooftheo1:eq2}
	\end{align} 
	Since $\lambda_c(\rho,\beta)<\zeta<1$,~\eqref{prooftheo1:eq1} and~\eqref{prooftheo1:eq2} imply that $q_{\zeta}$ has to  equal  $ \lambda_c(\rho,\beta) $.
	Therefore for  all $n$ large enough 
\begin{align*}	
&\theta_n(\lambda)\leq \exp\left(-c_{\lambda}n\right) \text{ for all }\lambda\in(0,\lambda_c(\rho,\beta)),\\
	&\theta(\lambda)\geq {u_{\zeta}}(\lambda-\lambda_c(\rho,\beta)), \text{ for all }\lambda\in(\lambda_c(\rho,\beta), \zeta).
\end{align*}
Since $\theta(\lambda)$  is non-decreasing in $\lambda$, by our choice of $\zeta$, we have 
$	\theta(\lambda)\geq \frac{u_{\zeta}}{2}(\lambda-\lambda_c(\rho,\beta))$ for all $\lambda\in(\lambda_c(\rho,\beta), 1)$,
	which concludes the proof of Proposition~\ref{theo1}.
%
%
%
%
\end{proof}
	
	Before we begin the proof of Lemma~\ref{lem2}, we do some housekeeping.
	Fix $n \geq 0$. We define a random variable $M_n$ as follows. For $\boldsymbol \omega$, the colouring of $[-n, n]^2$ is determined by $\Pc \cap [-n-R, n+R] \times (0, M_n(\boldsymbol{\omega})]$, i.e., the Poisson points  lying in the region $[-n-R, n+R] \times (0, M_n(\boldsymbol{\omega})]$ together with the associated $\boldsymbol \rho$,  $\boldsymbol \beta$,  $\boldsymbol \varepsilon$.

	Taking $m\in \N$  such that  $\P(\min\{\rho, \beta\} > 1/m)>\alpha$ for some $\alpha>0$ and dividing the region $[-n-R,n+R]^2$ into $16m^2(n+R)^2$ many distinct squares, each with a base size $\frac{1}{2m}\times \frac{1}{2m}$, we observe that if in each of these smaller squares there is a Poisson point  within   a height $h$ and for each such Poisson point $X_i$ (say), we have
	$\min\{\rho_i, \beta_i\} >  1/m$ then $M_n \leq h$. Thus
	\[
	\P(M_n \leq h)\geq 1-16m^2(n+R)^2 \exp \left(-\frac{\alpha h}{4m^2}\right). 
	\]
	For $\eta\in(0,1)$, let $h_{\eta}$ be the value of $h$ such that $16m^2(n+R)^2 \exp \left(-\frac{\alpha h}{4m^2}\right) = \eta$. 
	Thus $P(M_n \leq h_\eta)\geq 1- \eta$.
	
	Choosing  $\eta^{1/4}$ to be the reciprocal of an integer, we divide the region $[-n-R,n+R]^2\times (0,h_{\eta}]$ into distinct cylinders $C_1,C_2,\ldots C_k$, where $ k = 4\eta^{-1/2}(n+R)^2$, each with base size $\eta^{1/4}\times\eta^{1/4}$. 
	For $i\geq 1$,
	\begin{align*}
		& \P\left(\text{there exist two or more Poisson points in } {C}_i\right) \nonumber\\
		& = 1-\exp\left(-\eta^{1/2} h_{\eta}\right)-\exp\left(-\eta^{1/2} h_{\eta}\right)\cdot\eta^{1/2} h_{\eta} \nonumber\\
		& =O(\eta \log^2 \eta) \text{ as }\eta \to 0.
	\end{align*}
	
	For $n$ fixed as above, we define a local increasing event $A_n$ as   the event that {the origin is connected to $S_n$ via a red path and define $ B_n^{{\eta}} $ as the event that  the origin is connected to $S_n$ via a red path using Poisson points only in $\R^2\times(0,h_\eta]$.
		Let  $\theta_n^{\eta}(\lambda)=\P_\lambda(B_n^{\eta})$. 
		From our choice of $h_\eta$, as $ \eta \to 0$ we have $\theta_n^{\eta}(\lambda)= \theta_n(\lambda)+O(\eta)$; hence 
		\begin{equation}
			\label{r:thetaeta}
			\theta_n(\lambda)(1-\theta_n(\lambda)) = \theta_n^{\eta}(\lambda)(1-\theta^{\eta}_n(\lambda))+O(\eta).
		\end{equation}
		
		Noting that $\theta_n^{\eta}(\lambda)(1-\theta_n^{\eta}(\lambda)) = \text{Var}(1_{B_n^{{\eta}}})$, we want to apply the OSSS inequality for the function $f = 1_{B_n^{{\eta}}}$. We take $\Omega_i$ to be the restriction of the Poisson process $\Pc$ on the cylinder $C_i$ along with the associated $\boldsymbol {\rho}$, $\boldsymbol {\beta}$, $\boldsymbol {\varepsilon} $ for the  points of the Poisson process lying in the cylinder $C_i$. Also, $\text{Inf}_i$  and $\delta_i$ depends on $\eta$ and $C_i$ for all $i$ so we denote them by $\text{Inf}^{\eta}_{C_i}$ and $\delta^{\eta}_{C_i}$ respectively.

We need the following two technical results for the proof of Lemma~\ref{lem2}.
		\begin{lemma}
			\label{ddp>inf}
			\begin{equation*}
				\frac{d}{dp}\P_\lambda({A_n}) \geq \frac{1}{2}\cdot\limsup_{\eta\rightarrow 0} \sum_{i\geq 1} \text{Inf}_{{C}_i}^{\eta}(1_{B_n^{{\eta}}}).
			\end{equation*}
		\end{lemma}
		\begin{proof}
			Note that by construction,
			\[
			\text{Inf}_{{C}_i}^{\eta}(1_{B_n^{\eta}})= \P_\lambda \left(1_{B_n^{\eta}}(\boldsymbol{\omega},\boldsymbol{\varepsilon}) \neq 1_{B_n^{\eta}}(\boldsymbol{\omega_{[{C}_i]}},\boldsymbol{\varepsilon_{[{C}_i]}})\right),
			\]
			where $(\boldsymbol{\omega_{[V]}},\boldsymbol{\varepsilon_{[V]}})$ equals as $(\boldsymbol{\omega},\boldsymbol{\varepsilon})$ outside $V$ and is independently resampled in $V\subset\R^3$. Let $\boldsymbol{\varepsilon_{[V]}^{\text{red}}}$  equal $\boldsymbol{\varepsilon}$ outside $V$ and $\boldsymbol{1}$ inside $V$. Similarly let $\boldsymbol{\varepsilon_{[V]}^{\text{blue}}}$ equal $\boldsymbol{\varepsilon}$ outside $V$ and $\boldsymbol{0}$ inside $V$. Then we have
			\begin{align*}
				\text{Inf}_{{C}_i}^{\eta}(1_{B_n^{\eta}})
				&\leq 2\cdot\P_\lambda\left( 
				1_{B_n^{\eta}}\left(\boldsymbol{\omega},\boldsymbol{\varepsilon_{[{C}_i]}^{\text{blue}}}\right)
				\neq 1_{B_n^{\eta}}\left(\boldsymbol{\omega},\boldsymbol{\varepsilon_{[{C}_i]}^{\text{red}}}\right)
				\right)\\
				&\leq 2\cdot\P_\lambda\left( 
				1_{B_n^{\eta}}\left(\boldsymbol{\omega},\boldsymbol{\varepsilon_{[{C}_i]}^{\text{blue}}}\right)
				\neq 1_{B_n^{\eta}}\left(\boldsymbol{\omega},\boldsymbol{\varepsilon_{[{C}_i]}^{\text{red}}}\right),|\boldsymbol{\omega}\cap {C}_i|=1, M_n(\boldsymbol{\omega})\leq h_\eta 
				\right)\\
				&\hspace{10.25cm}+O(\eta \log^2 \eta)\\
				&\leq 2\cdot\P_\lambda\left( 
				1_{A_n}\left(\boldsymbol{\omega},\boldsymbol{\varepsilon_{[{C}_i]}^{\text{blue}}}\right)
				\neq 1_{A_n}\left(\boldsymbol{\omega},\boldsymbol{\varepsilon_{[{C}_i]}^{\text{red}}}\right),|\boldsymbol{\omega}\cap {C}_i|=1
				\right) +O(\eta \log^2 \eta)\\
				&\leq 2\cdot\E_\lambda\left[|\text{Piv}_{A_n}\cap {C}_i|\right]+O(\eta \log^2 \eta).
			\end{align*}
			Thus
			$$
			\limsup_{\eta\rightarrow 0} \sum_{i\geq 1} \text{Inf}_{{C}_i}^{\eta}(1_{B_n^{{\eta}}}) \leq 2\cdot\E_\lambda \left[|\text{Piv}_{A_n}|\right],
			$$
			and hence, by applying  Russo's formula, we get the required result.
		\end{proof}
		
		Let $P_i$ be the measure induced by $\P_\lambda$ when we restrict the Poisson point process $\Pc$ to $C_i$. For $c_i$  the centre of the projection $\pi({C}_i)$ of $C_i$ on its first two co-ordinates, let $E(c_i,S_k,\eta)$ be the event that   $c_i$ is connected to  $ S_k $  via a red path using Poisson points in $\R^2\times (0,h_{\eta}]$.  Then we have
		\begin{lemma}
			\label{lem3}
			There exists a constant $a>0$ such that for any $k\in\{1,2,\ldots,n\}$, there exists an algorithm $T_k$ determining $1_{{B_n^{{\eta}}}}$ with the property that for all $i\geq 1$, 
			\[
			\begin{split}
				\delta_{{C}_i}^{ \eta }(T_k)\leq a\bigotimes_{j\geq 1}P_j \left(E(c_i,S_k,\eta)\right).
			\end{split}
			\]
		\end{lemma}
		
		\begin{proof}
			Let 
			$$
			A^{(s)}=\{y\in\R^2:d(y,A)\leq s \}
			$$
			denote the fattening of  $A\subseteq \R^2$.  Here $d$ represents the $L_{\infty}$ distance. The algorithm $T_k$ is as follows
			\begin{itemize}
				\item[(a)] Take $\Ai_0=S_k$, $\Ci_0=\emptyset$.
				\item[(b)] For $t\geq 1$, $\Ci_t=\Ci_{t-1}\cup\left(\bigcup\left\{{C}_i:\pi({C}_i)\cap (\Ai_{t-1})^{(2R)}\neq \emptyset\right\}\right)$. Now reveal all the Poisson points in $\Ci_t$. Consider the colouring of $\R^2$ generated by the Poisson points in $\Ci_t$. Let $C(S_k, \Ci_t)$ be the union of $S_k$ and all red connected components of $(\Ai_{t-1})^{(R)}$ intersecting $S_k$ with respect to the colouring generated by Poisson points in $\Ci_t$. Take $\Ai_t= C(S_k, \Ci_t)\cap [-n,n]^2 $.
				\item[(c)] Stop the algorithm if any of the two conditions hold.
				\begin{enumerate}
					\item we get some $l$ such that ${\bf 0} \in \Ai_l$ and $\Ai_l\cap S_n\neq \emptyset$. In that case $1_{B_n^{\eta}}=1$.
					\item the previous condition does not hold, but we get some $l$ for which $\Ai_l=\Ai_{l+1}$. In that case $1_{B_n^{\eta}}=0$.
				\end{enumerate}
			\end{itemize}
			The revelation of ${C}_i$ implies the existence of  $l$ for which  $ d(\pi({C}_i),C(S_k,\Ci_l))\leq 2R $ and hence $ d(c_i,C(S_k,\Ci_l))\leq 2R+\frac{\eta^{1/4}}{2}<3R $. Now, if ${C}_i$ is revealed and $B(c_i,6R)$ is red then $c_i$ is connected to $S_k$ via a red path. So by the FKG inequality, we have 
			\begin{align*}
				&\bigotimes_{j\geq 1}P_j\left( E(c_i,S_k,\eta)\right)\\
				&\geq \delta_{{C}_i}^{ \eta }(T_k)\P_\lambda(B(0,6R)\text{ is red using only the Poisson points in $\R^2\times (0,h_{\eta}]$ }), 
			\end{align*}
			which implies $\delta_{{C}_i}^{\eta}(T_k)\leq a\bigotimes_{j\geq 1}P_j\left( E(c_i,S_k,\eta)\right)$
			for some constant $a$. This proves the lemma.
		\end{proof}

		Now we have all the machinery to prove Lemma~\ref{lem2}. 
		\begin{proof}[Proof of Lemma~\ref{lem2}]
			We first  note that
			the differentiability of $\theta_n(\lambda)$ was shown in Russo's formula.
			From the OSSS inequality, Lemma~\ref{lem3} and~\eqref{r:thetaeta}, we have
			\[
			\begin{split}
				\theta_n(\lambda)(1-\theta_n(\lambda))\leq   a\sum_{i\geq 1} \bigotimes_{j\geq 1}P_j\left(E(c_i,S_k,\eta)  \right)\mbox{Inf}_{{C}_i}^{ \eta }(1_{B_n^{{\eta}}}) +O(\eta),\\
			\end{split}
			\]
			which holds for all $k\in\{1,2,\ldots,n\}$, and so
			\[
			\begin{split}
				\theta_n(\lambda)(1-\theta_n(\lambda))\leq \frac{ a}{n}\sum_{i\geq 1}\sum_{k=1}^{n} \bigotimes_{j\geq 1}P_j\left(E(c_i,S_k,\eta)  \right)\mbox{Inf}_{{C}_i}^{\eta}(1_{B_n^{{\eta}}}) +O(\eta).
			\end{split}
			\]
			Since 
			\[
			\sum_{k=1}^{n} \bigotimes_{j\geq 1}P_j \left(E(c_i,S_k,\eta)  \right)\leq 2\sum_{k=0}^{n-1}\theta_k(\lambda),
			\]
			we have 
			\[
			\theta_n(\lambda)(1-\theta_n(\lambda))\leq\frac{2 a}{n}\left(\sum_{k=0}^{n-1}\theta_k(\lambda) \right)\left(\sum_{i\geq 1} \mbox{Inf}_{{C}_i}^{\eta}(1_{B_n^{\eta}}) \right) +O(\eta).
			\]
			Taking $\limsup$ as $\eta\rightarrow 0$ and using Lemma~\ref{ddp>inf}, we get
			\[
			\theta_n(\lambda)(1-\theta_n(\lambda))\leq\frac{4 a}{n}\left(\sum_{k=0}^{n-1}\theta_k(\lambda)\right)\frac{d}{d\lambda}\theta_n(\lambda).
			\]
			Now, $\theta_n(\lambda)$ is decreasing in $n$, so
			\[
			\theta_n(\lambda)\leq \frac{4 a}{n\left(1-\theta_1(\zeta) \right)}\left(\sum_{k=0}^{n-1}\theta_k(\lambda)\right)\frac{d}{d\lambda}\theta_n(\lambda),
			\]
			for all $\zeta>\lambda$ and hence Lemma~\ref{lem2} holds with $u_{\zeta}=\frac{1-\theta_1(\zeta)}{4a}$.
		\end{proof}

		
		Standard argument, see, for example, the proofs of Theorems~4.3 and~4.4 of Meester and Roy (1996), completes the proof of Theorem~\ref{th:dual}.
		
		\section{Covered Area Fraction}
		\label{Sec:caf}
		From Robbin's theorem (see Theorem~4.21 of Molchanov (2005)) we have
		\begin{align*}
			\text{CAF}_{\text{red}}(\lambda,\rho,\beta)
			&= \P_\lambda (\text {origin is red}).
		\end{align*}
		This characterisation gives us
		\begin{prop}
			\label{th:caf}
			For $\rho$ and $\beta$ two positive  random variables with finite second moments we have 
			\[
			\text{CAF}_{\text{red}}(\lambda,\rho,\beta) = \frac{\lambda\mathbb{E}[\rho^2]}{\lambda\mathbb{E}[\rho^2]+(1-\lambda)\mathbb{E}[\beta^2]}.
			\]
		\end{prop}

		\begin{proof}
			Suppose instead of taking a two coloured confetti model we consider an $n$-coloured model, i.e., we have  $n$ independent Poisson point processes $\{\Pc_i\}_{i=1}^n$ on $ \R^2 \times (0,\infty)$ with $n$ distinct colours.  For $i\in \{1,2,\ldots,n\}$ let $\lambda_i \geq 0$ with $\sum_{i=1}^n \lambda_i = 1$.  Assume that the Poisson point process $\Pc_i$ has intensity $\lambda_i$ and the associated squares have sides of fixed  length $r_i$. Without loss of generality we assume $r_1\leq r_2\leq\cdots\leq r_n$. Let $C_j := [-r_j, r_j]^2 \times (0, \infty)$. Let $l_1, l_2, \ldots$ be the ordering of the Poisson points of  $\cup_{i=1}^n \Pc_i$ lying in the cylinder $C_n$, where the ordering is done according to the value of the third co-ordinate. Thus  $l_k$ is the $k$th lowest Poisson point in the cylinder $C_n$. 
			
			For any $k \geq 1$, in order for the origin to be coloured due to $l_k$ and for it to have the $i^{\text{th}}$ colour, we must have 	
			\begin{itemize}
				\item[(i)] $l_k \in \Pc_i$ and $l_k\in C_r$;  an event with probability $\frac{\lambda_i r^2_i}{r^2_n}$,
				\item[(ii)] for all $1\leq m \leq k-1$, the point $l_m$ has the $j$th colour and lies in $C_n \setminus C_j$, for some $j \in \{i, \ldots , n\}$;  an event with probability $\left(\sum_{j=1}^{n}\frac{\lambda_j(r^2_n-r^2_j)}{r^2_n}\right)^{k-1}$ because of the independence of the Poisson processes.
			\end{itemize}
			The above two events being independent, we have that the covered area fraction, $\text{CAF}_i$, of the $i^{\text{th}}$ colour is given by
			$$
			\sum_{k=1}^{\infty} \mathbb{P}(\text{origin  has colour $i$ due to $l_k$}) =  \frac{\lambda_i r^2_i}{r^2_n}\sum_{k=1}^{\infty} \left(\sum_{j=1}^{n}\frac{\lambda_j(r^2_n-r^2_j)}{r^2_n}\right)^{k-1} 
			= \frac{\lambda_i r_i^2}{\sum_{j=1}^{n}\lambda_jr_j^2 }.
			$$

			Now, suppose we convert all the colours of $I\subseteq \{1,2,\ldots,n\}$ to red and the remaining colours to blue. Let $\Ri=\cup_{i\in I} \Pc_i$, 
			$\Bi=\cup_{i\notin I} \Pc_i$ and  $\lambda =\sum_{i\in I}\lambda_i$. For $u\in I$,  the random variable $\rho$ takes value $r_u$ with probability $\lambda_u/\lambda$, and for $v\notin I$,  the random variable $\beta$ takes value $r_v$ with probability $\lambda_v/(1- \lambda)$. Therefore 
			$$
			\text{CAF}_{\text{red}} (\lambda,\rho,\beta) = \sum_{i\in I}\text{CAF}_i
			= \frac{\sum_{i\in I}\lambda_ir_i^2}{\sum_{i\in I}\lambda_ir_i^2+\sum_{i\notin I}\lambda_ir_i^2 } 
			= \frac{\lambda\mathbb{E}[\rho^2]}{\lambda\mathbb{E}[\rho^2]+(1-\lambda)\mathbb{E}[\beta^2]}.
			$$
			
			A truncation argument, together with the dominated convergence theorem, completes the proof for general $\rho$ and $\beta$ with finite second moments.
		\end{proof}		
		For the proof of Theorem~\ref{th:cCAF1/2},
		we use the equivalent representation of the confetti model introduced prior to the statement of the transitivity condition.
		
	\begin{proof}[Proof of Theorem~\ref{th:cCAF1/2}]
		 \textit{(ii) $\Leftrightarrow$ (iii).}
			Recall, for this theorem, we assume $\rho$ and $\beta$ are fixed constants.
			From Proposition~\ref{th:caf}, 	
			\[
			\text{cCAF}_{\text{red}}(\rho,\beta)= \frac{\lambda_c(\rho,\beta)\rho^2}{\lambda_c(\rho,\beta)\rho^2 +(1-\lambda_c(\rho,\beta))\beta^2} = \frac{1}{2}
			\]
			if and only if $ \lambda_c(\rho,\beta)\rho^2=(1-\lambda_c(\rho,\beta))\beta^2$, i.e., 
			$\lambda_c(\rho,\beta) =\frac{\beta^2}{\rho^2+\beta^2}$.
			This proves the equivalence of (ii) and (iii).
			
		\textit{(ii) $\Rightarrow$ (i).} 
			Suppose (ii) holds and that red is supercritical in the confetti model $(\Ri, \lambda_r, \rho; \Bi, \lambda_b, \beta)$ and blue is supercritical in the confetti model $(\Bi, \lambda_b, \beta; \Gi, \lambda_g, \gamma )$. From Proposition~\ref{th:caf}, we have $\lambda_r \rho^2>\lambda_b\beta^2>\lambda_g\gamma^2$. Therefore $\text{CAF}_{\text{red}}(\frac{\lambda_r}{\lambda_r+\lambda_g},\rho,\gamma) > 1/2$, and hence red is supercritical in $(\Ri, \lambda_r, \rho;\Gi, \lambda_g, \gamma )$. Thus the transitivity condition holds.
			
		\textit{(i) $\Rightarrow$ (ii).}
			Finally, suppose the transitivity condition holds, and we have a confetti model $(\Ri, \lambda_r, \rho; \Bi, \lambda_b, \beta)$ with $\text{CAF}_{\text{red}}(\lambda,\rho,\beta) < 1/2$, i.e.,  $\lambda_r \rho^2<\lambda_b\beta^2$, but red is supercritical. Observe that by  the scaling $(x_1, x_2, x_3) \mapsto (x_1/\rho, x_2/\rho, x_3)$, the confetti model $(\Ri, \lambda_r, \rho; \Bi, \lambda_b, \beta)$ is equivalent to the model $(\Ri, \lambda_r\rho^2,1;\Bi,\lambda_b\rho^2,\frac{\beta}{\rho})$ in the sense that all percolation properties will be preserved. A further scaling $(x_1, x_2, x_3) \mapsto (x_1, x_2, \lambda_r \rho^2 x_3)$ shows that the above confetti model is equivalent to the model $(\Ri,1,1;\Bi,\mu,\nu)$, where $\mu=\lambda_b/\lambda_r$ and $\nu=\beta/\rho$. Note that $\sigma=\mu\nu^2>1$. By a similar scaling argument, we see that the model $(\Ri,1,1;\Bi,\mu,\nu)$ is equivalent to the models $(\Ri,\mu^{k-1},\nu^{k-1};\Bi,\mu^k,\nu^k)$ for any $k\in \N$. Thus red is supercritical in  $(\Ri, \lambda_r, \rho; \Bi, \lambda_b, \beta)$ implies that, for any $k\in\N$, red is supercritical in the confetti model $(\Ri,\mu^{k-1},\nu^{k-1};\Bi,\mu^k,\nu^k)$. Therefore, by the transitivity condition in (i), red is supercritical in the confetti model $(\Ri,1,1;\Bi,\mu^n,\nu^n)$ for any $n\in \N$.
			
			 It is clear that red cannot be  supercritical in the model $(\Ri,1,1;\Bi,\mu,\nu)$ when $\nu=1$ and $\mu>1$.  So we are left with two cases.
			
			\vspace{.3cm}
			
			\noindent \textit{Case I. $\nu>1$}:  Consider the model $(\Ri,1,1;\Bi,\mu^n,\nu^n)$. We divide $[0,\nu^n]\times[0,3\nu^n]$ into $12$ equal disjoint  squares, each with sides of length $\nu^{n}/2$. The probability that there exists a blue Poisson point within a height $h$ of each of these 12 small squares  is  
			$$
			\left( 1-e^{-\mu^n\nu^{2n}\frac{h}{4}}  \right)^{12} = \left( 1-e^{-\sigma^n\frac{h}{4}}  \right)^{12}.
			$$

			Now, consider the projection $\pi(\Ri\cap(\R^2\times[0,h_0]))$ of the red points in the region $\R^2\times[0,h_0]$ on the plane $\R^2 \times \{0\}$; here $\pi:\R^3\rightarrow\R^2$ is a projection onto first two co-ordinates. The points $\pi(\Ri\cap(\R^2\times[0,h_0]))$ form a Poisson point process of intensity $h_0$ on the plane. Associated with each of the points in the projection, we place a square with sides of unit length and  parallel to the axis; thereby we have a continuum Boolean percolation model  (Meester and Roy (1996)). We denote this model by $\left(\pi(\Ri\cap(\R^2\times[0,h_0])), h_0, 1\right)$.
			
			Choosing $h_0$ small enough, we may ensure that the continuum Boolean percolation model on the plane as obtained  above is in its subcritical phase. So there exists $N_0$ large enough,  such that for all $n\geq N_0$, the   probability that $[0,n]\times[0,3n]$ admits a red horizontal crossing in the continuum model is at most $\kappa_0/3 > 0$, where $\kappa_0$ is as in Proposition~\ref{th:exp}.	
			Choose $n_0$ large enough so that $\nu^{n_0}>N_0$ and also 
			\begin{equation}
				\label{r:h_oprob}
				\left( 1-e^{-\sigma^{n_0}\frac{h_0}{4}}  \right)^{12}\geq 1-\frac{\kappa_0}{3}.
			\end{equation}
			Now, the existence of blue Poisson points below a height $h_0$ and above each of the 12 squares, along with the non-existence of a red horizontal crossing in the projected continuum model guarantees the non-existence of  a red horizontal crossing of $[0,n_0]\times[0,3n_0]$ in the confetti setting. From~\eqref{r:h_oprob}, the probability that this occurs is larger than $1 - (\frac{\kappa_0}{3}+\frac{\kappa_0}{3})$; and hence from Proposition~\ref{th:exp}, red is not supercritical in the confetti model $(\Ri,1,1;\Bi,\mu^n,\nu^n)$ which contradicts our assumption.
			
			\vspace{.3cm}
			
			\noindent \textit{Case II. $\nu<1$}: 
			Again, consider the model $(\Ri,1,1;\Bi,\mu^n,\nu^n)$.
			Choose $h_0 > 0$ small enough such that 
			\begin{align}
				\label{r:vert}
				\P\left(\text{there exists a red Poisson point in } [0,1]\times[0,3]\times[0,h_0]\right) = 1-e^{-3h_0} < \frac{\kappa_0}{3},
			\end{align}
			where $\kappa_0 > 0$ is as above. Now, let
			\begin{align*}
				p_n:=\, & \P\left( \text{there exists a  vertical blue crossing of } [0,1]\times[0,3] \text{ in the blue }\right.\\
				& \left. \qquad \text{continuum Boolean model  }(\pi(\Bi\cap(\R^2\times[0,h_0])),h_0\mu^n,\nu^n)\right)
			\end{align*}
			
			The event $\{$there exists a  vertical blue crossing of $[0,1]\times[0,3] $ in the blue 
			continuum Boolean model  $(\pi(\Bi\cap(\R^2\times[0,h_0])),h_0\mu^n,\nu^n)\}$,
			together with the event $\{$there does not exist any red Poisson point in $[0,1]\times[0,3]\times[0,h_0]\}$, implies that there does not exist any  horizontal red crossing of $[0,1]\times[0,3]$ in the confetti  model.
			
			Assume, for the moment, 
			\begin{align}
				\label{r:p_n}
				p_{n_0} > 1- \frac{\kappa_0}{3} \text{ for some } n_0\geq 1,
			\end{align}
			then along with~\eqref{r:vert}, we have
			$$
			\P \left(\text{there exists a  horizontal red crossing of } [0,1]\times[0,3] \text{ in the confetti  model}\right) <
			\kappa_0
			$$
			which, along with Proposition~\ref{th:exp}, yields a contradiction.

			To obtain~\eqref{r:p_n}, we cover the rectangle $[0,1]\times[0,3]$ by $\lceil\frac{6}{\nu^n}\rceil\times\lceil\frac{2}{\nu^n}\rceil$ many small squares each with sides of length  $\nu^n/2 $. Observe that if within a height $h_0$ of each of these squares there is a blue Poisson point, then 
			the  continuum model $(\pi(\Bi\cap(\R^2\times[0,h_0])),h_0\mu^n,\nu^n)$ admits a 
			vertical blue crossing of $ [0,1]\times[0,3]$. 
			Thus we have
			$$
			p_n >   \left( 1-e^{-\sigma^n\frac{h}{4}}  \right)^{\frac{21}{\nu^{2n}}} .
			$$
			Now,
			\begin{align*}
				\lim\limits_{n\rightarrow\infty} \log\left( 1-e^{-\frac{\sigma^nh}{4}}  \right)^{\frac{21}{\nu^{2n}}}
				&\geq \lim\limits_{n\rightarrow\infty} \frac{\log\left( 1-e^{-\frac{\sigma^nh}{4}}  \right)}{e^{-\frac{\sigma^nh}{4}}}\cdot e^{-\frac{\sigma^nh}{4}}\left(\frac{21}{\nu^{2n}}\right)\\
				&= \lim\limits_{n\rightarrow\infty} \frac{21\cdot\log\left( 1-e^{-\frac{\sigma^nh}{4}}  \right)}{e^{-\frac{\sigma^nh}{4}}}\exp\left(-\frac{\sigma^nh}{4}-2n\log \nu\right) = 0.
			\end{align*}
			This validates~\eqref{r:p_n}.

			Therefore,   whenever $\lambda_r\rho^2<\lambda_b\beta^2$ red is not supercritical and similarly whenever $\lambda_r\rho^2>\lambda_b\beta^2$ blue is not supercritical. Hence we have $\lambda_r\rho^2=\lambda_b\beta^2$ at criticality, i.e., the critical covered area fraction equals $1/2$.
			
			This completes the proof.
		\end{proof}
		
		Before we conclude this section, we present a sufficient condition for the transitivity condition to hold.
		\begin{prop}
			\label{r:suff}
			Consider the confetti model $(\Ri, 1,1; \Bi,\frac{\sigma}{\beta^2}, \beta)$ and let $S_n := \{z\in \mathbb{R}^2: \|z\|_{\infty}=n\}$. For fixed $\sigma > 1$, 
			the transitivity condition holds whenever 	$\P(C^{\text{red}}({\bf 0}) \cap S_n \neq \emptyset)$ is 
			\begin{itemize}
			\item[(a)] monotonic (either non-increasing or non-decreasing)  for $\beta\in (0,1)$ and 
			\item[(b)] monotonic (either non-increasing or non-decreasing)  for $\beta\in (1,\infty)$ for each $n\geq 1$.
			\end{itemize}
		\end{prop}
		
		\begin{remark}
			\label{rem:suff}
			Note that in view of Proposition~\ref{r:suff}, it is enough to show that $\P(C^{\text{red}}({\bf 0}) \cap S_n \neq \emptyset)$ is monotonic (either non-inccreasing or non-decreasing)  in $\beta$ on $(0,1)$ and
			also monotonic (either non-inccreasing or non-decreasing) in $(1,\infty)$ for each $n\geq 1$. Let $\boldsymbol{x}$ be the realisation of the Poisson point process $\Pc$. We know that for any fix $\boldsymbol x$ and $\alpha\in(\beta/2,2\beta)$, the occurrence of the event 
			$$\{ C^{\text{red}}({\bf 0}) \cap S_n \neq \emptyset \mbox{ in the model $(\Pc, \frac{\alpha^2}{\sigma+\alpha^2},1,  \alpha)$ } \}$$
			depends only on the 
			Poisson points with indices $1,2, \ldots,k(\boldsymbol{x},\beta)  $ for some $k(\boldsymbol{x},\beta)\in\N$ depending only on $\boldsymbol{x}$ and $\beta$.
			Let $\boldsymbol{\beta}=(\beta_i)_{i\in\N} $ with $\beta_i=\beta$ for $i>k(\boldsymbol{x},\beta)$. Then 
			\begin{align*}
				\frac{d}{d\beta}\P_{\boldsymbol{\beta}}(C^{\text{red}}({\bf 0}) \cap S_n \neq \emptyset)
				=\sum_{i=1}^{k(\boldsymbol{x},\beta)} \frac{\partial}{\partial \beta_i} \P_{\boldsymbol{\beta}}(C^{\text{red}}({\bf 0}) \cap S_n \neq \emptyset)
			\end{align*}
			Here, $\P_{\boldsymbol{\beta}}$ indicates that the probability of the $i$-th Poisson point being blue (respectively, red)  is $\frac{\sigma}{\sigma+\beta_i^2}$ (respectively,  $\frac{\beta_i^2}{\sigma+\beta_i^2}$) and if it is blue (respectively, red) the associated blue (respectively, red) square will have side length $\beta_i$ (respectively, $1$). We fix $\delta \in(0,\beta/2)$ and write $\boldsymbol{\beta}_i=(\beta_1,\beta_2,\ldots, \beta_{i-1},\beta_i+\delta,\beta_{i+1},\ldots,\beta_{k(\boldsymbol{x},\beta)},\beta,\beta,\ldots)$.
			We couple the models associated  with $\boldsymbol{\beta}$ and $\boldsymbol{\beta}_i$  in such a way  that the corresponding realisations of the Poisson point processes match and if $\boldsymbol{\varepsilon}$ and $\boldsymbol{\varepsilon'}$ are the corresponding Bernoulli sequences, then   $\varepsilon_i\leq\varepsilon_i'$ and $\varepsilon_j=\varepsilon_j'$ for all $j\neq i$. 
			Let $A$ be the event $\left\{ C^{\text{red}}({\bf 0}) \cap S_n \neq \emptyset \right\}$. Then for any $i\leq k(\boldsymbol{x},\beta)  $,
			\begin{align*}
				&	\P_{\boldsymbol{\beta}_i}(A)-\P_{\boldsymbol{\beta}}(A)\\[.25cm]
				=\,\,&	\P_{\boldsymbol{\beta}_i}(A \text{ occurs}, \varepsilon_i=\varepsilon_i'=1 )-\P_{\boldsymbol{\beta}}(A \text{ occurs}, \varepsilon_i=\varepsilon_i'=1)\\
				&+	\P_{\boldsymbol{\beta}_i}(A \text{ occurs}, \varepsilon_i=0,\varepsilon_i'=1 )-\P_{\boldsymbol{\beta}}(A \text{ occurs}, \varepsilon_i=0,\varepsilon_i'=1)\\
				&+	\P_{\boldsymbol{\beta}_i}(A \text{ occurs}, \varepsilon_i=\varepsilon_i'=0 )-\P_{\boldsymbol{\beta}}(A \text{ occurs}, \varepsilon_i=\varepsilon_i'=0)\\[.25cm]
				=\,\,& 0+\left(\frac{\sigma}{\sigma+\beta_i^2}-\frac{\sigma}{\sigma+(\beta_i+\delta)^2}\right)\cdot \P_{\boldsymbol{\beta}}(\text{$i$-th Poisson point is `colour pivotal' for }A )\\
				\\&-\frac{\sigma}{\sigma+(\beta_i+\delta)^2}\cdot\P_{\boldsymbol{\beta}}(\text{$i$-th Poisson point is `size pivotal' for }A ).
			\end{align*}
			Here, the $i$-th Poisson point is `colour pivotal' for $A$ indicates if we change its colour from red to blue, it will affect the occurrence of $A$. Similarly, the $i$-th Poisson point is `size pivotal' for $A$ means if the point is blue and we change the side length of the associated blue square  from $\beta_i$ to $\beta_i+\delta$ , it will affect the occurrence of $A$.
			
			So, we need to compare these two types of pivotality to calculate the sign of $\frac{d}{d \beta} \P_{\boldsymbol{\beta}}(A)$. And if we can show that $\frac{d}{d \beta} \P_{\boldsymbol{\beta}}(A)$ does not change sign in $(0,1)$ and also in $(1,\infty)$, then that, together with Proposition~\ref{r:suff},  will imply the `transitivity' condition.

		\end{remark}

		\begin{proof}[Proof of Proposition~\ref{r:suff}] We first recall some properties of the continuum Boolean model of percolation (see Meester and Roy (1996)). Let $(\Ci, \lambda, 1)$ be a $2$-dimensional continuum Boolean model with squares of sides of length $1$ and let $\lambda_c^B(1)$ be its critical intensity.  The critical CVF for $(\Ci, \lambda,1)$ is  $1 - \exp(-\lambda_c^B(1))$, and for percolation of fixed sized squares the model is subcritical or supercritical according as 
			the CVF is smaller or larger than $1 - \exp(-\lambda_c^B(1))$.

			Consider the confetti percolation model $(\Ri,1,1;\Bi,\frac{\sigma}{\beta^2},\beta)$ for fixed $\sigma>1$. Let $h\in(\frac{\lambda_c^B(1)}{\sigma},\lambda_c^B(1))$ and consider the slab  $\mathbf{S_1}=\R^2\times (0,h)$. 
			
			The  projection of only the red squares in the slab $\mathbf{S_1}$ of $(\Ri,1,1;\Bi,\frac{\sigma}{\beta^2},\beta)$  onto the plane $\R^2\times \{0\}$  yields a continuum Boolean model  $(\Ci_R, h, 1)$ of intensity $h$ and hence is subcritical by our choice of $h$.  Let $W_R$ ($V_R$) denote the occupied (vacant) region of $(\Ci_R, h,1)$. 
			Now Theorem 3.5 of Meester and Roy (1996) guarantees the existence of  $N>1$ such that  $(\Ci_R, h, 1)$ admits a horizontal occupied crossing of the rectangle $R_N$ has a probability at most ${\kappa_0}/{4} $, where $\kappa_0$ is as in Lemma~3.3 of Meester and Roy (1996).  Thus
			\begin{equation}
				\label{r:recross}
				P(E_{n_0}) \geq 1- \kappa_0 /2 \text{ for some } n_0 \geq 1
			\end{equation}
			where, for $n \geq 1$, $E_n$ is the event that in the model $(\Ci_R, h,1)$,  
			\begin{itemize}
				\item[(i)] there is a vacant vertical crossing of the rectangle $R_N :=[0,N]\times[0,3N]$,
				\item[(ii)] distinct connected occupied components of $W_R$ in $R_N$ are separated by a distance at least ${1}/{n}$, and
				\item[(iii)] if a connected component of $W_R$ does not intersect the left or right boundary of $R_N$, then it is at a distance at least ${1}/{n}$ from the boundary.
			\end{itemize}
			
			Again, the projection of only the blue squares in the slab $\mathbf{S_1}$ of $(\Ri,1,1;\Bi,\frac{\sigma}{\beta^2},\beta)$  onto the plane $\R^2\times \{0\}$  yields a continuum Boolean model  $(\Ci_B, \frac{\sigma h}{\beta^2},\beta)
			$ which is independent  of $(\Ci_R, h, 1)$. Moreover, $(\Ci_B, \frac{\sigma h}{\beta^2},\beta)
			$  is supercritical by our choice of $h$. Let
			$W_B^{(\beta)}$ ($V_B^{(\beta)}$) denote the occupied (vacant) region of $(\Ci_B, \frac{\sigma h}{\beta^2},\beta)
			$.
			From Theorems~4.3 and~4.4 together with Lemma~4.1 of Meester and Roy (1996),  we have for the model $(\Ci_B, \frac{\sigma h}{\beta^2},\beta)
			$, there exist constants $C_1, C_2 > 0$, independent of $a$, such that 
			$$
			\P\left(\text{diam}(V_B^{(\beta)}([0,\beta]^2) \geq a \beta \right) \leq C_1 \exp(-C_2 a).
			$$
			The above equation may be rewritten as
			\begin{equation}
				\label{r:scale}
				\P\left(\text{diam}(V_B^{(\beta)}([0,\beta]^2) \geq b \right) \leq C_1 \exp(-C_2 b/\beta) \text{ for any } b > 0.	
			\end{equation}

			Now, suppose that $E_{n_0}$ occurs and that 	$W_R\cup V_B^{(\beta)}$ admits a horizontal crossing of the rectangle $R_N$. It must then be the case that there exists a square $S$ (say) of size $\beta\times\beta$ for which the diameter of  $V_B^{(\beta)}(S)$ is at least ${1}/{n_0}$. Also, $R_N$ can be tiled with atmost $\big\lceil \frac{3N^2}{\beta^2}\big \rceil$ squares of size $\beta\times\beta$. Thus, from~\eqref{r:recross} and~\eqref{r:scale}, we have that there exists $\beta_1 > 0$ such that,
			\begin{align} \label{referee:beta1}
				&\P(W_R\cup V_B^{(\beta)} \text{ admits a horizontal crossing of the rectangle } R_N) \nonumber\\
				&\leq \frac{\kappa_0}{2}+ \left\lceil \frac{3N^2}{\beta^2}\right \rceil\cdot c_1\exp\left(-c_2\cdot\frac{1}{n_0\beta} \right)\nonumber\\
				& < \frac{3\kappa_0}{4} \text{ for all } 0 < \beta\leq\beta_1.
			\end{align}
			Hence, for all $0 < \beta\leq\beta_1$, from Proposition~\ref{th:exp}, we see that, because 
			$$
			\P( R_N \text{ admits a red horizontal crossing}) < 3\kappa_0/4,
			$$
			there is no red percolation in  $(\Ri,1,1;\Bi,\frac{\sigma}{\beta^2},\beta)$.
			

			Similarly for the confetti model $(\Ri,1,1;\Bi,\frac{1}{\sigma\beta^2},\beta)$ taking
			$h'\in(\lambda_c^B(1),\sigma\lambda_c^B(1))$ and the slab  $\mathbf{S_2}:=\R^2\times (0,h')$. The projection of the red squares in this slab on $\R^2 \times \{0\}$ yields a supercritical continuum Boolean model 
			$(\Ci_R, h', 1)$ and the projection of the blue squares in this slab on $\R^2 \times \{0\}$ yields  a subcritical 
			continuum Boolean model $(\Ci_B, \frac{h'}{\sigma\beta^2},\beta)$. Note here we have used the scale equivalence of the models  $(\Ci_B, \frac{h'}{\sigma\beta^2},\beta)$ and $(\Ci_B, {h'}/{\sigma},1)$.

			Calculations similar to the previous case gives us that there exist $M > 0$ and  $\beta_0 > 0$ such that,
			\begin{align}\label{referee:beta0}
				&\P(W_B^{\beta}\cup V_R \text{ admits a horizontal crossing of the rectangle } R_M) \nonumber\\
				&\leq \frac{\kappa_0}{2}+ \left\lceil \frac{3M^2}{\beta^2}\right \rceil\cdot c_1\exp\left(-c_2\cdot\frac{1}{n_0\beta} \right)\nonumber\\
				& < \frac{3\kappa_0}{4} \text{ for all } 0 < \beta\leq\beta_0 
			\end{align}
			and, for all $0 < \beta\leq\beta_0$, in the confetti percolation model $(\Ri,1,1;\Bi,\frac{1}{\sigma\beta^2},\beta)$,
			$$
			\P( R_M \text{ admits a blue horizontal crossing}) < 3\kappa_0/4.
			$$
			
			So interchanging the parameters, we obtain that for all $0 < \beta\leq\beta_o$, in the confetti percolation model $(\Ri,\frac{1}{\sigma\beta^2},\beta;\Bi, 1,1)$,
			$$
			\P( R_M \text{ admits a red horizontal crossing}) < 3\kappa_0/4.
			$$
			Also scaling shows that the model $(\Ri,\frac{1}{\sigma\beta^2},\beta;\Bi, 1,1)$ is equivalent to the model $(\Ri,1,1;\Bi,\sigma\beta^2,\frac{1}{\beta})$. Thus Proposition~\ref{th:exp} yields that there is no red percolation in  $(\Ri,1,1;\Bi,\frac{\sigma}{\gamma^2},\gamma)$ for $\gamma > 1/\beta_0$.
			
			Thus, from (\ref{referee:beta1}) and (\ref{referee:beta0}), we  have $\beta_1 \leq1$  and $\beta_2:=1/\beta_0 \geq 1$ such that, for all $\beta\notin(\beta_1,\beta_2)$ red does not percolate. The monotonicity hypothesis in the statement of the proposition guarantees that, for the confetti percolation model $(\Ri,1,1;\Bi,\frac{\sigma}{\beta^2},\beta)$, with $\sigma > 1$,
			$$
			\P(\text{diam} (C^{\text{red}}({\bf 0})) = \infty) =0 \text{ for all } \beta> 0.
			$$
			Now any confetti model with covered area fraction of red strictly less than $1/2$ is scale equivalent to a model $(\Ri,1,1;\Bi,\frac{\sigma}{\beta^2},\beta)$. Thus the covered area fraction of red strictly less than $1/2$ implies  red does not percolate. This completes the proof of the proposition.
		\end{proof}

		\section{Proof of Theorem~\ref{th:cCAFrandom}}
		\label{Sec:caf_random}
		
		To prove Theorem~\ref{th:cCAFrandom}, we use the following technical result. 
		Here we use the equivalent representation of the confetti model that was introduced before the statement of the transitivity condition. $\lambda_c^B(1)$ denotes the critical intensity for percolation of the continuum Boolean model with squares of sides of length $1$. 
		\begin{lemma}
			\label{lem:caf_random}
			Suppose we have a confetti percolation model $(\Ri,\lambda_1,\alpha_1;\Bi, \lambda_2,\alpha_2)$ with $\alpha_1,\alpha_2\leq R$. Let $N>R$ and $\delta\in(0,\kappa_0)$, where $N$ and $\kappa_0$ are as in Proposition~\ref{th:exp}. Let $G$ be the event that there exists a point in $R_N := [0,N] \times  [0,3N]$ with  no blue square within  a height $1$ above it and $H_N$  be the event that  $R_N$ admits a  red horizontal crossing. Suppose the following two conditions hold:		
			\begin{enumerate}
				\item $\mathbb{P}_{\lambda_1,\alpha_1, \lambda_2,\alpha_2} (G)<\delta $
				\item $ \mathbb{P}_{\lambda_1,\alpha_1, \lambda_2,\alpha_2}\left(H_N\cap G^{\complement}\right)<\kappa_0-\delta $
			\end{enumerate}		
			then there exist $\lambda_3$ and $ \alpha_3 $ with $ \alpha_3\leq R $ such that 		
			\[
			\lambda_3\mathbb{E}\left[\alpha_3^2\right]=\lambda_1\mathbb{E}\left[\alpha_1^2\right]+\frac{\lambda_c^B(1)}{2} 
			\]		
			and the above two conditions hold for the confetti percolation model \linebreak $(\Ri,\lambda_3,\alpha_3; \Bi,\lambda_2,\alpha_2)$.
		\end{lemma}
	We first use the above lemma to prove the theorem.
		\begin{proof}[Proof of Theorem~\ref{th:cCAFrandom} using Lemma~\ref{lem:caf_random}]
			Let $N$, $\delta$ and $G$ be as in the previous lemma, and fix $0 < r < R$.  Choose $\mu$ large enough such that 	in the confetti model $(\Ri,0,r;\Bi,\mu,r)$,
			$$
			\mathbb{P}_{0,r, \mu,r} (G)<\delta .
			$$	
			Observe that the probability that there exists a horizontal red crossing of $R_N$ in the confetti model $(\Ri,0,r;\Bi,\mu,r)$ is 0. So the assumptions in Lemma~\ref{lem:caf_random} hold. Let $q\in \mathbb{N}$  be such that
			\begin{equation}
				\label{r:q}
				\frac{q\cdot\frac{\lambda_c^B(1)}{2}}{q\cdot\frac{\lambda_c^B(1)}{2} +\mu r^2}>t.
			\end{equation}
			Applying Lemma~\ref{lem:caf_random}  $q$ times, we obtain $\lambda$ and $\rho$ such that $\rho\leq r$ and 
			$ \lambda\mathbb{E}\left[\rho^2\right]=q\cdot\frac{\lambda_c^B(1)}{2} $. Also, in the confetti model $(\Ri,\lambda,\rho;\Bi,\mu,r)$, we have
			\begin{itemize}
				\item[(i)] $\mathbb{P}_{\lambda,\rho, \mu,r} (G)<\delta	$ and
				\item[(ii)] $\mathbb{P}_{\lambda,\rho, \mu,r} (H(R_N)\cap G^{\complement}) < \kappa_0-\delta$,
			\end{itemize}
			which yields, by Proposition~\ref{th:exp},
			$$\P_{\lambda,\rho, \mu,r} (\text{diam} (C^{\text{red}}({\bf 0}))\geq a)\leq c_1\exp(- c_2a),
			$$
			i.e.,  red does not  percolate in the confetti model $(\Ri,\lambda,\rho;\Bi,\mu,r)$. However, from~\eqref{r:q}, the covered area fraction of red for the above model is greater than $t$.
		\end{proof}
	To complete the proof of Theorem~\ref{th:cCAFrandom}, we now prove Lemma~\ref{lem:caf_random}. The idea of the proof is 
	the same as that of Meester, Roy and Sarkar (1994), viz. to create ``dust-like'' particles which contribute to the covered area fraction, but do not connect so as to percolate.
		\begin{proof}[Proof of Lemma~\ref{lem:caf_random}]
			Let $H_0(R_N)$ be the event that  $R_N$ admits a  red horizontal crossing in the confetti model $(\Ri,\lambda_1,\alpha_1;\Bi, \lambda_2,\alpha_2)$. Let $\Ri'$ be an  independent Poisson process on $\R^2 \times (0, \infty)$ of intensity $m^2\bar{\lambda}$ and at each point of this process we centre a square with sides of length $1/m$ where  $m\in \mathbb{N}$ and $\bar{\lambda} =\frac{\lambda_c^B(1)}{2}$. The superposed model  $(\Ri\cup\Ri',\mu_m,\varrho_m;\Bi, \lambda_2,\alpha_2)$, has intensity $\mu_m :=\lambda_1+m^2\bar{\lambda}$ and squares with sides of length given by the random variable 
			\[
			\varrho_m :=\left\{ \begin{array}{ll}
				\alpha_1&\mbox{ with probability }\frac{\lambda_1}{\mu_m}\\
				\frac{1}{m}&\mbox{ with probability }1-\frac{\lambda_1}{\mu_m}.
			\end{array} \right.
			\]
			
			We consider the colouring of $\R^2$ obtained in two distinct ways. 
			First, we colour $\R^2$ by the superposed model $(\Ri\cup\Ri',\mu_m,\varrho_m;\Bi, \lambda_2,\alpha_2)$.  Let $H_m(R_N)$ be the event that  $R_N$ admits a red  horizontal crossing in this colouring. Second, we colour $\R^2$ by the confetti model $(\Ri,\lambda_1,\alpha_1;\Bi, \lambda_2,\alpha_2)$ and, on it, we superpose 
			$\left(\pi\left(\Ri'\cap(\R^2\times[0,1])\right),m^2\bar{\lambda},\frac{1}{m}\right) $, the continuum Boolean model 
			arising as a projection of $\Ri'\cap(\R^2\times[0,1])$. Thus an earlier blue point of $\R^2$ may become red because of a square on it coming from the process $\pi\left(\Ri'\cap(\R^2\times[0,1])\right)$. Let $\tilde{H}_m(R_N)$ be the event that  $R_N$ admits a horizontal red crossing in this superposed model.	 
			
			Noting that on $G^{\complement}$ the points of the processes above a height $1$ do not affect  the colouring of $R_N$, we have $H_m(R_N)\cap G^{\complement}\subseteq \tilde{H}_m(R_N)\cap G^{\complement}$.
			
			For any $n\geq 1$, let $F_n$ be the event that in the  confetti percolation model $(\Ri,\lambda_1,\alpha_1;\Bi, \lambda_2,\alpha_2)$,
			\begin{itemize}
				\item[(i)] $H_0(R_N)$ does not occur,
				\item[(ii)] any two red components in $R_N$ is separated by a distance at least $1/n$, and
				\item[(iii)] any red component that does not touch either the left or the right boundary of $R_N$ is at a distance at least $1/n$ away from them.
			\end{itemize}
			Since  $F_n^{\complement}\to H_0(R_N) $ as $n \to \infty$ and  $\mathbb{P}(H_0(R_N)\cap G^{\complement})<\kappa_0-\delta$, we can get $n_0$ such that 
			\begin{equation}
				\label{r:Gc}
				\mathbb{P}(F_{n}^{\complement}\cap G^{\complement})<\kappa_0-\delta \text{ for all } n \geq n_0.
			\end{equation}
			
			Let $B(y,t)=\left\{z\in \mathbb{R}^2 \,:\, \|z-y\|_{\infty}\leq \frac{t}{2} \right\}$. Also, take $C_m(B(0,s))$ to be the union of all the connected occupied components of the continuum Boolean percolation model $ \left(\pi\left(\Ri'\cap(\R^2\times[0,1])\right),m^2\bar{\lambda},\frac{1}{m}\right) $ which have non-empty intersection with  $B(0,s)$.
			Since $\bar{\lambda} <\lambda_c^B(1)$, from Theorems 4.3 and 4.4 together with Lemma 4.1 of Meester and Roy (1996),  we have
			\[
			\mathbb{P}(\text{diam}\left(C_m(B(0,1/m))\right)>1/(2n_0))\leq c_1\exp( - c_2m/(2n_0))
			\]		 
			for some constants $c_1,c_2$ depending only on $\bar{\lambda}$.	
			
			Now, tile $R_N$  by  $3N^2m^2$ many small squares  with sides of length $\frac{1}{m}$. Label them as $W_1,W_2,\ldots$ and let 		
			\[
			A^m_{n_0}=\bigcup_{i=1}^{3N^2m^2}\left\{\text{diam}(C_m(W_i))>1/(2n_0)\right\}.
			\]
			
			Observe that $\tilde{H}_{m_0}(R_N)\subseteq F_{n_0}^{\complement}\cup A^{m_0}_{n_0}$, and so
			$H_{m_0}(R_N)\cap G^{\complement}\subseteq \left(F_{n_0}^{\complement}\cap G^{\complement}\right)\cup A^{m_0}_{n_0}$. Thus
			$
			\mathbb{P}\left(H_{m_0}(R_N)\cap G^{\complement} \right) < \kappa_0 - \delta \text{ for $m_0$ large enough},
			$
			where we have used~\eqref{r:Gc} and the fact that 
			$\mathbb{P}\left(A^m_{n_0}\right)\leq 3N^2m^2c_1e^{-c_2\frac{m}{2n_0}}$.
			
			Taking  $\lambda_3=\mu_{m_0}$	and	$\alpha_3=\varrho_{m_0}$, we have 	
			\begin{itemize}
				\item[(i)] $\displaystyle \alpha_3<R$, 
				\item[(ii)] $\displaystyle\lambda_3\mathbb{E}\left[\alpha_3^2\right]= \lambda_1\mathbb{E}\left[\alpha_1^2\right]+m_0^2\bar{\lambda}\cdot\frac{1}{m_0^2}= \lambda_1\mathbb{E}\left[\alpha_1^2\right]+\frac{\lambda_c^B(1)}{2}$,
				\item[(iii)] $\displaystyle \mathbb{P}_{\lambda_3,\alpha_3, \lambda_2,\alpha_2}\left(H_N\cap G^{\complement}\right)=\mathbb{P}\left(H_{m_0}(R_N)\cap G^{\complement} \right)< \kappa_0-\delta$,
				\item[(iv)] $\displaystyle \mathbb{P}_{\lambda_3,\alpha_3, \lambda_2,\alpha_2} (G)=\mathbb{P}_{\lambda_1,\alpha_1, \lambda_2,\alpha_2} (G)<\delta$,
			\end{itemize}
			which completes the proof of the lemma.	
		\end{proof}

		\section{Voronoi percolation}
		\label{Voronoi}
		
		The Voronoi model of percolation is another model with the self-dual property. In the same vein as the confetti percolation model, we set up the model with `varying speeds' as follows:
		For $\lambda \in [0,1]$, let 
		\begin{itemize}
			\item[(i)] $\Pc:= (X_1, X_2, \ldots) $  be a Poisson process of intensity $1$ on $ \R^2 $,
			\item[(ii)] $\boldsymbol {\varsigma} := (\varsigma_1, \varsigma_2, \ldots)$ be  a collection of i.i.d. positive, bounded random variables,
			\item[(iii)] $\boldsymbol {\tau} := (\tau_1, \tau_2, \ldots)$ be a collection of i.i.d. positive, bounded  random variables,
			\item[(iv)] $\boldsymbol {\varphi} := (\varphi_1, \varphi_2, \ldots )$  be a collection of i.i.d. Bernoulli random variables taking values $1$ and $0$ with probabilities $\lambda$ and 
			$1- \lambda$ respectively.
		\end{itemize}
		We assume the four processes above are independent of each other.
		
	Let 
			\begin{align*}
				{\bf C}_i  &
				:= \left\{x: \frac{||x- X_i||_2}{\varsigma_i 1_{\{\varphi _i = 1\}} + \tau_i 1_{\{\varphi _i = 0\}}} \leq \frac{||x- X_j||_2}{\varsigma_j1_{\{\varphi _j = 1\} }+ \tau_j1_{\{\varphi _j = 0\}}} \text{ for all } j \neq i\right\} ,\\
				{\bf S}_i & := \begin{cases} C_i & \text{ if } \varphi_i = 1\\
					\emptyset &  \text{ if } \varphi_i = 0\end{cases}\\
				{\bf T}_i & := \begin{cases} C_i & \text{ if } \varphi_i = 0\\
					\emptyset &  \text{ if } \varphi_i = 1\end{cases}
		\end{align*} 
		Thus $\R^2$ is partitioned into a `shaded' region (i.e., union of sets of the form ${\bf S}_i$) and a `tiled' region (i.e., union of sets of the form ${\bf T}_i$). We denote this model by $({\cal P}, \lambda, \varsigma,\tau)$.  
		
		Let $C^S$ and $C^T$ be the (random) shaded/tiled regions of $\R^2$ in this model. Voronoi percolation studies questions of shaded/tiled percolation on the $x-y$ plane $\R^2$, i.e., for what values of the parameters 
		$\lambda$, $\varsigma$,  $\tau$ do we have an unbounded connected shaded component of $C^S$ or an unbounded connected tiled component of $C^T$ (or possibly both/neither).
		
		Taking  $C^S({\bf 0})$ and $C^T({\bf 0})$ to be, respectively,  the maximal connected components of $C^S$ and $C^T$ containing the origin ${\bf 0}$, let
		$$
		\lambda_v (\varsigma,\tau) := \inf\{\lambda : \P_\lambda(\text{diam} (C^S({\bf 0})) = \infty) > 0\}.$$
		
		When $\varsigma = \tau = \text{ constant}$,  Bollob\'{a}s and Riordan (2006a) showed that $\lambda_v = 1/2$ (note here that in this case, $\lambda_v$ does not depend on the common constant value taken by $\varsigma$ and $\tau$).
		
		When $\varsigma$ and $\tau$ are constants not necessarily the same,  Kira, Neves and Schonmann (1998) conjectured that  $\lambda_v(\varsigma, \tau) = \frac{\tau^2}{\varsigma^2+\tau^2}$ for all possible $\varsigma$ and $\tau$.

		For this Voronoi setup, we obtain results similar to those for the confetti percolation model. The following theorem was proved by Ahlberg, Tassion and Texeira (2018)  (see Theorem~8.2) via a Russo-Seymour-Welsh method, and in addition, they obtained that at criticality neither regions percolate.
		\noindent Based on methods developed by Duminil-Copin, Raoufi and Tassion (2019a), we have
		
			\begin{theorem}
				\label{r-vthm1}
				For $\varsigma$ and $\tau$ bounded random variables, we have
				$$
				\lambda_v(\varsigma,\tau) + \lambda_v(\tau, \varsigma) =  1.
				$$
				
				In particular, when $\varsigma \stackrel{d}{=} \tau$,
				$$
				\lambda_v(\varsigma,\tau) =1/2.
				$$
			\end{theorem}
			To prove the above, we first note that a sharp threshold result for  the Voronoi percolation model akin to  Proposition~\ref{theo1} for  the confetti percolation model follows immediately on changing the definition of $E_z$ of Duminil-Copin, Raoufi and Tassion (2019b) (page 487) to be  that $E_z$  is the event the $\eta_b$ 
			does not intersect the Euclidean ball of radius $( ||x-z||_2 -3\sqrt{2}) \frac{r}{R}$ around $z$.
			Now, the argument to show Theorem~\ref{th:dual} may be repeated to prove Theorem~\ref{r-vthm1}.

			For the Voronoi percolation model, the transitivity condition of Section~\ref {Sec:intro} is 
		
			\begin{align}
				\label{r-VC}
				\left(\frac{1}{\lambda_v(\varsigma,\tau)}-1\right)\left(\frac{1}{\lambda_v(\tau, \delta)}-1\right)\leq \frac{1}{\lambda_v(\varsigma, \delta)}-1,
			\end{align}
			for $\varsigma$, $\tau$ and $\delta$ constants, not necessarily equal. 
		
		As earlier, we define the covered area fraction as
		$\text{CAF}_{\text{shaded}}(\lambda, \varsigma,\tau)$ for the Voronoi model where 
		$$
		\text{CAF}_{\text{shaded}}(\lambda, \varsigma,\tau) := \E_\lambda\left[\ell(\text{shaded region in } [0,1]^2)\right].
		$$
		The critical covered area fraction is defined as 
		$$
	\text{cCAF}_{\text{shaded}}(\varsigma,\tau) := \E_{\lambda_v(\varsigma,\tau)}\left[\ell(\text{shaded region in } [0,1]^2)\right].
		$$
		From Theorem~\ref{r-vthm1},
		we have that $\text{cCAF}_{\text{shaded}}(\varsigma,\tau) = 1/2$ for $\varsigma \stackrel{d}{=} \tau$, both bounded and also bounded away from zero.

		Also, using a sharp threshold result akin to Proposition~\ref{theo1} and through an analysis of the covered area fraction, we have for the Voronoi percolation model:
		
		\begin{theorem}
			\label{th:cCAF1/3}
			The following are equivalent:
			\begin{itemize}
				\item[(i)] The transitivity condition~\eqref{r-VC} holds.
				\item[(ii)] For any $\varsigma$, $\tau$ positive constants, the critical covered area fraction for $(\Pc, \lambda, \varsigma, \tau)$ equals $1/2$.
				\item[(iii)] For any $\varsigma$, $\tau$ positive constants in the Voronoi model $(\Pc,  \lambda,\varsigma,   \tau)$,  the shaded region percolates if  $\lambda> \frac{\tau^2}{\varsigma^2 + \tau^2}$ and the tiled region percolates if  $\lambda < \frac{\tau^2}{\varsigma^2 + \tau^2}$, i.e., 
				$\lambda_v(\varsigma,\tau) = \frac{\tau^2}{\varsigma^2 + \tau^2}$.
			\end{itemize}
		\end{theorem}  
		In case the transitivity condition holds, the above theorem provides us  exact values of the critical parameter $\lambda_v(\varsigma,\tau)$ as conjectured by Kira, Neves and Schonmann (1998). A sufficient condition analogous to Proposition~\ref{r:suff}, as well as the subsequent discussion in Remark~\ref{rem:suff}, also applies to the Voronoi percolation model.
		
		The proofs of Theorem~\ref{v:cCAFrandom} and Proposition~\ref{CAFVoronoi}  (stated below)	follow, with minor changes, from the proofs of 
Theorem~\ref{th:cCAFrandom}	and 	Proposition~\ref{th:caf} and, as such,  we omit them here.
				\begin{theorem}
			\label{v:cCAFrandom}
			For any $t \in (0,1)$, there exists a Voronoi model $(\Pc, \lambda, \varsigma, \tau)$, with $\varsigma$ and $\tau$ random,   for which $CAF_{\text{shaded}}(\lambda,\varsigma,\tau) < t$ but the shaded region  percolates.
		\end{theorem}

				\begin{prop}
			\label{CAFVoronoi}
			For $\varsigma$ and $\tau$  two positive random variables with finite second moments, we have 
			\[
			\text{CAF}_{\text{shaded}}(\lambda,\varsigma,\tau) = \frac{\lambda\mathbb{E}[\varsigma^2]}{\lambda\mathbb{E}[\varsigma^2]+(1-\lambda)\mathbb{E}[\tau^2]}.
			\]
		\end{prop}

		\begin{proof}[Proof of Theorem~\ref{th:cCAF1/3}]
			The proofs of \textit{(ii) $\Leftrightarrow$ (iii)}  and \textit{(ii) $\Rightarrow$ (i)}  follow on the same lines as those in the proof of Theorem~\ref{th:cCAF1/2}.
			
			To prove \textit{(i) $\Rightarrow$ (ii)} we  will use the following result of Kira, Neves and Schonmann (1998):		
			\begin{prop}
				\label{KNS}
				For the Voronoi percolation model $(\Si,\alpha,1;\Ti,1-\alpha,\frac{1}{t})$,
				there exist positive  constants $0 <  c \leq C < \infty$, such that for small	enough $ t $,
				\begin{itemize}
					\item[(i)] the shaded region  percolates if $ 1-\alpha< c t^2 $, and
					\item[(ii)] the tiled region percolates if $ 1-\alpha> C t^2 $.
				\end{itemize}
			\end{prop}
			
			First we show that (ii) of Theorem~\ref{th:cCAF1/3}
holds under the transitivity condition. In case it does not, then there exist $\gamma_1$, $v_1$, $\gamma_2$, $v_2$ such that $\gamma_1v_1^2<\gamma_2v_2^2$ but the shaded region  is supercritical in the model $(\Si,\gamma_1,v_1;\Ti,\gamma_2,v_2)$. As in the confetti percolation, here also only the ratio of the growth speeds and the ratio of the intensities play a role in percolation. Let $v=v_2/v_1$ and  $\gamma=\gamma_2/\gamma_1$. Clearly, we have $\gamma v^2 >1$. A similar argument as in the case of confetti percolation gives us  that  this model is equivalent to the model
			$(\Si,1,1;\Ti,\gamma,v)$ which, in turn,  is  equivalent to the model $(\Si,\gamma^{k-1},v^{k-1};\Ti,\gamma^k,v^k)$ for any $k\in \N$. Therefore, the supercriticality of  the shaded region in the Voronoi model $(\Si,1,1;\Ti,\gamma^k,v^k)$ follows from the transitivity condition.
			It is clear that the shaded region  cannot be supercritical if $v=1$ and $\gamma>1$. So we need to consider the two cases $v>1$, $\gamma<1$ and $v<1$, $\gamma>1$.
			
			\vspace{.3cm}
			
			\textit{Case I. $\gamma<1$, $v>1$}: Since $\gamma v^2>1$, we have
			\[
			\lim_{n \to \infty} \frac{\gamma^nv^{2n}}{\gamma^n+1}=\infty,
			\]
			so there exists $k_1\in\N$ such that
			\[
			\frac{\gamma^{k_1}}{\gamma^{k_1}+1}>\frac{C}{v^{2{k_1}}}, \text{ where } C \text{ is as in Proposition~\ref{KNS}}.
			\]
			Therefore, the tiled region percolates in the model $(\Si,\frac{1}{1+\gamma^{k_1}},1;\Ti,\frac{\gamma^{k_1}} {1+ \gamma^{k_1}},v^{k_1})$. This  leads to a contradiction because, by scaling, the models $(\Si,\frac{1}{1+\gamma^{k_1}},1;\Ti,\frac{\gamma^{k_1}} {1+ \gamma^{k_1}},v^{k_1})$ and $(\Si,1,1;\Ti,\gamma^{k_1},v^{k_1})$ are equivalent.
			
			\vspace{.3cm}
			
			\textit{Case II. $\gamma>1$, $v<1$}: Consider the model $(\Si,1,1;\Ti,\frac{1}{\gamma^n},\frac{1}{v^n})$. Since $\gamma v^2>1$, we have
			\[
			\lim_{n \to \infty}\frac{\frac{1}{\gamma^n}}{1+\frac{1}{\gamma^n}}\cdot\frac{1}{v^{2n}}=\lim_{n \to \infty}
			\frac{1}{v^{2n}+\gamma^nv^{2n}}=0.
			\]
			So there exists $k_2\in\N$ such that
			\[
			\frac{\frac{1}{\gamma^{k_2}}}{1+\frac{1}{\gamma^{k_2}}}<cv^{2k_2}, \text{ where }  c \text{ is as in Proposition~\ref{KNS}}.
			\]
			
			Therefore the shaded region  percolates in the model $(\Si,1,1;\Ti,\frac{1}{\gamma^{k_2}},\frac{1}{v^{k_2}})$. Equivalently, the tiled region percolates in the model $(\Si,1,1;\Ti,\gamma^{k_2},v^{k_2})$, which leads to a contradiction. 
		\end{proof}

		\begin{proof}[Proof of  Theorem~\ref{v:cCAFrandom}]
			For the proof of  Theorem~\ref{v:cCAFrandom}, we need a result for the Poisson Boolean  model which is of independent interest. We relegate   its proof to the appendix.
			
			%

			Let $(\Xi_1,\lambda_1, \rho^{(1)} )$ and $(\Xi_2, \lambda_2, \rho^{(2)})$ be two independent Poisson Boolean models with $ 0 <  \rho^{(1)},  \rho^{(2)} \leq R$ and  $C^{(1)}$, $C^{(2)}$  their respective covered regions. Let $\Ci = C^{(1)} \setminus C^{(2)}$ and $\Di = \R^2 \setminus \Ci$. Let $R_N=[0,N]\times [0,3N]$.
			\begin{prop}
				\label{r-Bool1}
				
				(i) There exists $\kappa_0 > 0$ such that, in the superposition of the Boolean models $\Xi_1$ and $\Xi_2$, if, 
				for some $N > R$,
				$$
				\P\left(\text{there is a component of $\Di\cap R_N$ connecting the left and right edges of $R_N$}\right) < \kappa_0
				$$
				then, for $\Di({\bf 0})$ the connected component of $\Di$ containing the origin $\bf 0$, we have
				$$
				\P(\text{diam}(\Di({\bf 0})) \geq a) \leq K_1 \exp(-K_2 a)
				$$
				for all $a > 0$ and  positive constants $K_1$ and $K_2$ depending only on $\kappa_0$.
				
				(ii) Also, for all $\lambda < \lambda_c^B(1)$, there exists a Poisson Boolean model
				$(\Xi_3, \frac{\lambda}{v^2}, v)$ for some $v \in (0,1)$ (independent of the other two processes) with $C^{(3)}$ being its covered region such that in the superposition of the Boolean models $\Xi_1$, $\Xi_2$ and $\Xi_3$,
				taking $\Ei = \Di \cup C^{(3)}$, with $\Ei ({\bf 0})$, the connected component of $\Ei$ containing the origin $\bf 0$, we have
				$$
				\P(\text{diam}(\Ei({\bf 0})) \geq a) \leq K_1 \exp(-K_2 a)
				$$
				for all $a > 0$ and positive constants  $K_1$ and $K_2$ depending only on $\kappa_0$.
			\end{prop}

			To prove Theorem~\ref{v:cCAFrandom}, let $\Si$ be a Poisson point process with intensity $\lambda>\lambda_c^B(1)$.	Define	\[
			\Vi_0=\left(\cup_{x\in\Si}B(x,1)\right)^{\complement},
			\]
			We know that 
			$$
			\P(\text{diam}(\Vi_0({\bf 0})) \geq a) \leq K_1 \exp(-K_2 a)
			$$
			for all $a > 0$, where $\Vi_0({\bf 0})$ denotes the connected component of $\Vi_0$ containing the origin and  $K_1$ and $K_2$ are as in Proposition~\ref{r-Bool1}. Choose $q\in\N$ large enough such that 
			\[
			\frac{\lambda}{\lambda+q\cdot\frac{\lambda_c^B(1)}{2}}<t.
			\]
			Now applying Proposition~\ref{r-Bool1}  $q$ times, we get $v_1,v_2,\ldots, v_q\in(0,1)$ and a Poisson point process $\Ti_q$ with intensity $\mu_q=\sum_{i=1}^{q}\frac{\lambda_c^B(1)}{2v_i^2}$. Define
			\[
			\Vi_q=\left(\cup_{y\in\Ti}B(y,\tau^{(q)}_y)\right)\cup\left(\cup_{x\in\Si}B(x,1)\right)^{\complement},
			\]
			where $\tau^{(q)}$ takes value $v_i$ with probability $\frac{\lambda_c^B(1)}{2v_i^2\mu_q}$. Clearly,
			$$
			\P(\text{diam}(\Vi_q({\bf 0})) \geq a) \leq K_1 \exp(-K_2 a)
			$$
			for all $a > 0$, where $\Vi_q({\bf 0})$ denotes the connected component of $\Vi_q$ containing the origin and  $K_1$ and $K_2$ are as in Proposition~\ref{r-Bool1}.		This implies that $\R^2\setminus\Vi_q$ has an unbounded connected component with probability $1$.
			
			Now, consider the Voronoi percolation model $(\Si,\lambda,1;\Ti_q,\mu_q,\tau^{(q)})$. 
			For any $y\in \R^2\setminus\Vi_q$, there exists $X_i\in\Si$ such that 
			$
			||y-X_i||_2\leq \frac{1}{2}
			$, also, $
			||y-X_j||_2\geq \frac{\tau_j}{2}
			$ for all $X_j\in \Ti_q$,
			thus $\R^2\setminus\Vi_q\subseteq{C^S}$; which implies that  $C^S$ also has an unbounded connected component with probability $1$.
			However, the covered area fraction in the $\Si,\Ti^{(q)}$ model equals
			\[
			\frac{\lambda}{\lambda+\mu_q\E[{\tau^{(q)}}^2]}=\frac{\lambda}{\lambda+q\cdot\frac{\lambda_c^B(1)}{2}}<t,
			\]
			which completes the proof. 
		\end{proof}
		
		\section{Appendix}
		\begin{proof} [Proof of Proposition~\ref{th:exp}]
			Consider the integer lattice $\Zz$  with vertices $u$ and $v$ adjacent if and only if $\|u-v\|_{\infty}=1$. For $z=(z_1,z_2)\in\mathbb{Z}^2$, consider the squares 
			\begin{align*}
				D_0(z)&=\left(z_1N,(z_1+1)N\right]\times\left(z_2N,(z_2+1)N\right]\\
				D_1(z)&=\left((z_1-1)N,(z_1+2)N\right]\times\left((z_2-1)N,(z_2+2)N\right].
			\end{align*}		
			A vertex $z$ is open if  there exists a connected component $\Lambda \subseteq C^{\text{red}}$  such that $\Lambda\cap D_0(z)\neq\emptyset $ and also $\Lambda\cap D_1(z)^{\complement}\neq \emptyset$. Otherwise $z$ is closed. Thus
			\[
			q:=\P(\text{$z$ is open})\leq 4 \P(H_N)
			\]
			Since $N\geq R$, the event that $z$ is open depends only on the Poisson points in $\left((z_1-2)N,(z_1+3)N\right]\times\left((z_2-2)N,(z_2+3)N\right]\times (0,\infty)$. So, we have a dependent site percolation model with $ \P(\text{$z$ is open}) =q$ and the states of the vertices $w$ and $z$ are independent whenever $\|w-z\|_{\infty}\geq 5$.  Note that diam$\left( D_0(z)\right)= \sqrt{2}N$ and diam$\left( D_1(z)\right)= 3\sqrt{2}N$. Therefore $\text{diam} (C^{\text{red}}({\bf 0}))\geq a$ implies $C^{\text{red}}({\bf 0})$ must intersect at least $\lceil\frac{a}{\sqrt{2}N}\rceil$ many distinct translates of  $D_0(z)$, and if $a\geq 3\sqrt{2}N $ then $C^{\text{red}}({\bf 0})$ also intersects the boundary of the  correspondingly translated $D_1(z)^{\complement}$.
			Let $C_{{\text{site}}}$ denote the open cluster of the origin in the site percolation model. Then for $a/N\geq 3\sqrt{2}$,   $\text{diam} (C^{\text{red}}({\bf 0}))\geq a$     implies that $C_{{\text{site}}}$ consists of  at least $\lceil\frac{a}{\sqrt{2}N}\rceil$ many vertices. Now, given a set $I_n$ of $n$ vertices of $\Zz$, there are at least $\frac{n}{11^2} $ vertices whose states are independent of each other, since the states of two vertices $w$ and $z$ are independent whenever $\|w-z\|_{\infty}\geq 5$. So $\mathbb{P}$(all vertices of $I_n$ are open) $ \leq q^{\frac{n}{11^2}} $. Now, the total number all connected sets $I_n$ of $n$ vertices of $\Zz$ containing the origin is at most $(9e)^n$ (Kesten 1982, Lemma 5.1). Using this bound, we get 
			\[
			\P\left(\left|C_{\text{site}}\right|=n\right) \leq (9e)^n q^{\frac{n}{11^2}} \leq (9e)^n\left(4\P\left(H_N\right)\right)^{\frac{n}{11^2}}.
			\]
			Therefore taking $ \frac{a}{N}\geq 3\sqrt{2} $, we have 
			\begin{align*}
				\P(\text{diam} (C^{\text{red}}({\bf 0}))\geq a)&\leq \sum_{n\geq \lceil\frac{a}{\sqrt{2}N}\rceil}\P\left(\abs{C_{{\text{site}}}}=n\right)\\
				&\leq  \sum_{n\geq \lceil\frac{a}{\sqrt{2}N}\rceil} (9e)^n \left(4 \P\left(H_N\right)\right)^{\frac{n}{11^2}}.
			\end{align*}
			Choosing $\kappa_0<\frac{1}{4\cdot(9e)^{11^2}}$, we have that, whenever $\P\left(H_N\right)<\kappa_0$ and $ a/N\geq 3\sqrt{2} $,		
			\[
			\mathbb{P}(\text{diam} (C^{\text{red}}({\bf 0}))\geq a)\leq b_1e^{-b_2\frac{a}{N}} 
			\]		
			for positive constants $b_1$, $b_2$ depending only on $\kappa_0$. \end{proof}
		
		\vspace{.5cm}
		
		\begin{proof}[Proof of Proposition~\ref{r-Bool1}(i)]
			Consider the integer lattice $\Zz$  with vertices $u$ and $v$ adjacent if and only if $\|u-v\|_{\infty}=1$. For $z=(z_1,z_2)\in\mathbb{Z}^2$, consider the squares 
			\begin{align*}
				D_0(z)&=\left(z_1N,(z_1+1)N\right]\times\left(z_2N,(z_2+1)N\right]\\
				D_1(z)&=\left((z_1-1)N,(z_1+2)N\right]\times\left((z_2-1)N,(z_2+2)N\right].
			\end{align*}		
			A vertex $z$ is open if  there exists a connected component $\Lambda \subseteq {\Di}$  such that $\Lambda\cap D_0(z)\neq\emptyset $ and also $\Lambda\cap D_1(z)^{\complement}\neq \emptyset$. Otherwise $z$ is closed. Thus
			\begin{align*}
				q&:=\P(\text{$z$ is open})\\
				&\leq 4 \P\left(\text{there is a component of $\Di\cap R_N$ connecting the left and right edges of $R_N$} \right)
			\end{align*}	

			Now, an argument using the dependent site percolation structure (as in the proof of Proposition~\ref{th:exp}), yields
			\[
			\mathbb{P}(\text{diam} ({\Di}({\bf 0}))\geq a)\leq b_1e^{-b_2\frac{a}{N}} 
			\]		
			for positive constants $b_1$, $b_2$ depending only on $\kappa_0$. \end{proof}
		
		\vspace{.5cm}
		
		\begin{proof}[Proof of Proposition~\ref{r-Bool1}(ii)]
			Note that there exists $N>1$ such that $\R^2 \setminus \Di$ admits a vertical crossing of the rectangle $R_N$ with probability at least $1-\frac{\kappa_0}{4}$.
			Thus 
			\begin{equation}
				\P(E_{n_0}) \geq 1- \kappa_0 /2 \text{ for some } n_0 \geq 1
			\end{equation}
			where, for $n \geq 1$, $E_n$ is the event that
			\begin{itemize}
				\item[(i)]  $R_N \cap (\R^2 \setminus \Di)$ admits a vertical crossing of the rectangle $R_N$, 
				\item[(ii)] distinct connected  components of $\Di\cap R_N$ are separated by a distance at least $1/n$, and
				\item[(iii)] if a connected component of $\Di$ does not intersect the left or right edge of $R_N$ then it is at a distance at least $1/n$ from the boundary.
			\end{itemize}
			The Boolean model $(\Gi, \frac{\lambda}{v^2},v)$ being subcritical, there exist constants $C_1, C_2>0$ such that
			\[
			\P\left(\text{diam}(C_{\Gi}^{(v)}([0,v]^2) \geq a v \right) \leq C_1 \exp(-C_2 a).
			\]
			The above equation may be rewritten as
			\[
			\P\left(\text{diam}(C_{\Gi}^{(v)}([0,v]^2) \geq b \right) \leq C_1 \exp(-C_2 b/v) \mbox{ for any $b>0$}.
			\]	
			Now, suppose that $E_{n_0}$ occurs and that 	$\Di$ admits a horizontal crossing of the rectangle $R_N$. It must then be the case that there exists a square $S$ (say) of size $v\times v$ for which the diameter of  $C_{\Gi}^{(v)}(S)$ is at least $\frac{1}{n_0}$. Also, $R_N$ can be tiled with atmost $\left\lceil \frac{3N^2}{v^2}\right \rceil$ squares of size $v\times v$. So we have that there exists $v_0$ such that
			\begin{align*}
				&\P\left(\Ei \text{ admits a horizontal crossing of the rectangle } R_N\right) \\
				&\leq \frac{\kappa_0}{2}+ \left\lceil \frac{3N^2}{v^2}\right \rceil\cdot C_1\exp\left(-C_2\cdot\frac{1}{n_0 v} \right)\\
				& < \frac{3\kappa_0}{4} \text{ (for all $v$ small enough). }
			\end{align*}
			Hence using Proposition~\ref{r-Bool1}\,(i), we get that 
			\[
			\P\left(\text{diam}(\Ei({\bf 0})) \geq a \right) \leq K_1 \exp(-K_2 a) \mbox{ for any $a>0$},
			\]
			where $K_1, K_2$ are as in Proposition~\ref{r-Bool1}.
		\end{proof}

		\section*{Acknowledgement}
		The authors thank the referees whose suggestions led to a much better presentation of this paper. PPG thanks the Indian Statistical Institute for support for his doctoral progamme. 
		RR acknowledges the Grant MTR/2017/000141 from DST which
		supported this research.

	\end{document}